\documentclass[final,reqno,a4paper]{amsart}

\usepackage[foot]{amsaddr}
\usepackage{graphicx,enumerate,nicefrac,bm}
\usepackage{cite}
\usepackage{stmaryrd,amsmath, amssymb, amstext, amsfonts, mathrsfs}
\usepackage{algorithmicx,algpseudocode}
\usepackage{epstopdf}
\usepackage{a4wide}
\usepackage{color}

\newcommand{\dd}{\mathsf{d}}
\newcommand{\norm}[1]{\left\|#1\right\|}

\newcommand{\abs}[1]{\left|#1\right|}
\newcommand{\F}{\mathsf{F}}
\newcommand{\NF}{\mathsf{N}_{\F}}
\renewcommand{\d}{\mathsf{d}}

\newcommand{\dx}{\,\d x}
\newcommand{\jmp}[1]{\left\llbracket#1\right\rrbracket}
\newcommand{\NN}[1]{\left\|#1\right\|}

\newcommand{\hp}{{\rm hp}}

\renewcommand{\u}{u^{\rm hp}}
\newcommand{\uu}{u^{(\Delta t,\rm hp)}_{n+1}}
\newcommand{\T}{\mathcal{T}}
\newcommand{\p}{\bm p}
\newcommand{\V}{V^{\rm hp}(\T,\p)}
\newcommand{\e}{\epsilon}

\newcommand{\NNN}[1]{\left|\!\left|\!\left|#1\right|\!\right|\!\right|}
\newcommand{\intp}{\pi_{\V}}
\newcommand{\eremk}{\hbox{}\hfill\rule{0.8ex}{0.8ex}}

\newtheorem{theorem}{Theorem}[section]

\newtheorem{proposition}[theorem]{Proposition} 
\newtheorem{corollary}[theorem]{Corollary}

\theoremstyle{definition}
\newtheorem{example}[theorem]{Example}

\newtheorem{algorithm}[theorem]{Algorithm}
\newtheorem{remark}[theorem]{Remark}

\newcommand{\myState}[1]{\State\parbox[t]{\dimexpr\linewidth-\algorithmicindent}{#1\strut}}

\numberwithin{equation}{section}

\numberwithin{equation}{section}

\title[$hp$-Adaptive Newton-Galerkin FEM]{An $hp$-Adaptive Newton-Galerkin Finite Element Procedure for Semilinear Boundary Value Problems}

\author[M.~Amrein]{Mario Amrein}
\address{Lucerne University of Applied Sciences and Arts, CH-6002 Luzern, Switzerland}
\email{mario.amrein@hslu.ch}

\author[J.~M.~Melenk]{Jens M.~Melenk}
\address{
Institut f\"ur Analysis und Scientific Computing, 
TU Wien, A-1040 Wien, Austria}
\email{melenk@tuwien.ac.at}

\author[T.~P.~Wihler]{Thomas P.~Wihler}
\address{Mathematics Institute, University of Bern, CH-3012 Switzerland}
\email{wihler@math.unibe.ch}
\thanks{TW was supported by the Swiss National Science Foundation (SNF), Grant No. 200021-162990.}

\begin{document}

\begin{abstract}
In this paper we develop an $hp$-adaptive procedure for the numerical solution of general, semilinear elliptic boundary value problems in 1d, with possible singular perturbations. Our approach combines both a prediction-type adaptive Newton method and an $hp$-version adaptive finite element discretization (based on a robust {\em a posteriori} residual analysis), thereby leading to a fully $hp$-adaptive Newton-Galerkin scheme. Numerical experiments underline the robustness and reliability of the proposed approach for various examples.
\end{abstract}



\maketitle

\section{Introduction}

The purpose of this paper is the development of a numerical approximation procedure for semilinear elliptic boundary value problems, with possible singular perturbations. More precisely, for a fixed parameter~$\e>0$ (possibly with~$\e\ll 1$), and a continuously differentiable function $f:\,\Omega\times\mathbb{R}\to\mathbb{R}$, we consider the problem of finding a function~$u:\,\Omega\to\mathbb{R}$ in an interval~$\Omega=(a,b)$, $a<b$, which satisfies
\begin{equation}\label{eq:poisson}
\begin{aligned}
-\e u''(x) &=f(x,u(x)) \text{ for }x\in \Omega,\qquad
u(a)=u(b)=0. 
\end{aligned}
\end{equation}
Solutions of~\eqref{eq:poisson} are typically not unique (even infinitely many solutions may exist), and, in the singularly perturbed case, may exhibit boundary layers, interior shocks, and (multiple) spikes. The existence of multiple solutions due to the nonlinearity of the problem and/or the appearance of singular effects constitute two challenging issues when solving problems of this type numerically; see, e.g.,\cite{RoStTo08,Verhulst}.

\emph{Newton-Galerkin Methods:} In order to address the numerical solution of nonlinear singularly perturbed problems of type~\eqref{eq:poisson}, a Newton-Galerkin approach has recently been proposed in~\cite{AmreinWihler:14,AmreinWihler:15}. It is built upon two main ideas: Firstly, adaptive Newton methods are applied on the continuous level in order to approximate the nonlinear differential equation~\eqref{eq:poisson} by a sequence of linearized problems. Secondly, the resulting linear problems are discretized by means of an adaptive finite element procedure which, in turn, is based on an $\e$-robust \emph{a posteriori} error analysis. Then, providing an appropriate interplay between the adaptive Newton method and the adaptive finite element approach, by which we either perform a Newton step (if the Newton linearization effect dominates) or refine the current finite element mesh based on some suitable {\em a posteriori} error indicators (in case that the finite element discretization causes the main source of error), a fully adaptive Newton-Galerkin scheme is obtained; see also the related works on monotone nonlinear problems~\cite{ChaillouSuri:07,El-AlaouiErnVohralik:11,CongreveWihler:15}, or the articles~\cite{ChaillouSuri:06,Han:94} dealing with modelling errors in linearized models. Our numerical results in~\cite{AmreinWihler:15} reveal that sensible solutions can be found even in the singularly perturbed case, and that our scheme is reliable in the sense that the adaptive Newton-Galerkin method is able to converge to solutions which preserve the qualitative structure of (reasonably chosen) initial guesses (see also~\cite{AmreinWihler:14,ScWi11}).

\emph{$hp$-FEM and $hp$-adaptivity:} The aim of the present paper is to extend our work in~\cite{AmreinWihler:15} from a low-order to an $hp$-adaptive Newton-Galerkin approach. For simplicity of exposition, we focus on the 1d case, however, we mention that an extension to the multi-d case is possible as well. Exploiting the $hp$-framework may potentially lead to \emph{exponential rates of convergence} in the numerical approximation. For the purpose of~$hp$-adaptivity, we carry out an $\e$-robust $hp$-version \emph{a posteriori} residual analysis for the FEM discretizations of the Newton linearizations. We remark that using $hp$-adaptivity, in contrast to the adaptive low-order methodology developed in the previous works~\cite{El-AlaouiErnVohralik:11,AmreinWihler:14,AmreinWihler:15}, introduces an additional aspect of adaptivity that may further improve the performance of the Newton-Galerkin approach. In the $hp$-literature, several adaptive strategies and algorithms have been proposed; see, e.g., the overview article~\cite{MiMc11} and the references therein. In particular, we point to the smoothness estimation techniques proposed in~\cite{eibner-melenk06,HoustonSeniorSuliENUMATH,HoustonSuliHPADAPT,Ma94}, and the related approach~\cite{FaWiWi14,Wi11_2,Wi11} involving continuous Sobolev embeddings. The latter works will be applied in the current paper in order to drive the $hp$-adaptive finite element refinements; in combination with the adaptive Newton linearizations this will lead to a fully $hp$-adaptive Newton-Galerkin procedure.

\emph{Problem formulation and notation:} Henceforth, we suppose that a (not necessarily unique) solution~$u\in X:=H^1_0(\Omega)$ of~\eqref{eq:poisson} exists; here, we denote by $H^1_0(\Omega)$ the standard Sobolev space of functions in~$H^1(\Omega)=W^{1,2}(\Omega)$ with zero trace on~$\partial\Omega:=\{a,b\}$. Furthermore, signifying by~$X'=H^{-1}(\Omega)$ the dual space of~$X$, and upon defining the map $\F_{\e}: X\rightarrow X'$ through
\begin{equation}\label{eq:Fweak}
\left \langle \F_{\e}(u),v\right \rangle :=  \int_{\Omega}\left\{\e u'v'-f(u)v\right\}\dx\qquad \forall v\in X,
\end{equation}
where $\left\langle\cdot,\cdot\right\rangle$ is the dual product in~$X'\times X$, the above problem~\eqref{eq:poisson} can be written as a nonlinear operator equation in~$X'$:
\begin{equation}\label{eq:F0}
u\in X:\qquad \F_\e(u)=0.
\end{equation}
In addition, for open subsets~$D\subseteq\Omega$, we introduce the norm
\begin{equation*}
\NNN{u}_{\e,D}:=\Bigl(\e\norm{u'}_{0,D}^2 +\norm{u}_{0,D}^2 \Bigr)^{\nicefrac{1}{2}},
\end{equation*}
where~$\|\cdot\|_{0,D}$ denotes the $L^2$-norm on~$D$. Note that, in the case of~$f(u)=-u+g$, with given~$g\in L^2(\Omega)$, i.e., when~\eqref{eq:poisson} is linear and strongly elliptic, the norm $\NNN{\cdot}_{\e,\Omega}$ is a natural energy norm on~$X$. Frequently, for~$D=\Omega$, the subindex~`$D$' will be omitted.  Furthermore, the associated dual norm of~$\F_\e$ from~\eqref{eq:Fweak}  is given by
\[
\NNN{\F_\e(u)}_{X',\e} =  \sup_{\genfrac{}{}{0pt}{}{v\in X}{\NNN{v}_{\e}=1}}\int_{\Omega}\left\{\e u'v'-f(u)v\right\}\dx.
\]
Throughout this work we shall use the abbreviation $x \preccurlyeq y$ to mean $x \leq c y$, for a constant $c>0$ independent of the mesh size~$h$, the local polynomial degree~$p$, and of~$\e>0 $. 

\emph{Outline:} In Section~\ref{sc:anewton} we will consider the Newton method within the context of dynamical systems in general Banach spaces, and present a prediction strategy for controlling the Newton step size parameter. Furthermore, Section~\ref{sc:FEM} focuses on the application of the Newton methodology to semilinear elliptic problems as well as on the discretization of the problems under consideration by $hp$-finite element methods; in addition, we derive an $\e$-robust {\em a posteriori} residual analysis. An $hp$-adaptive procedure and a series of numerical experiments illustrating the performance of the proposed idea will be presented in Section~\ref{sc:adaptivity}. Finally, we summarize our findings in Section~\ref{sc:concl}.

\section{Adaptive Newton Methods in Banach Spaces}\label{sc:anewton}

In the following section we shall briefly revisit an adaptive prediction-type Newton algorithm from~\cite{AmreinWihler:15}. 

\subsection{Abstract Framework}
Let $ X, Y $ be two Banach spaces, with norms~$\|\cdot\|_X$ and~$\|\cdot\|_Y$, respectively. Given an open subset~$\Xi\subset X$, and a (possibly nonlinear) operator~$\F:\,\Xi\to Y$, we consider the {\em nonlinear} operator equation
\begin{equation}\label{eq:F=0}
\F(u)=0,
\end{equation} 
for some unknown zeros~$u\in\Xi$. Supposing that the Fr\'echet derivative~$\F'$ of~$\F$ exists in~$\Xi$ (or in a suitable subset), the classical Newton method for solving~\eqref{eq:F=0} starts from an initial guess~$u_0\in\Xi$, and generates the iterates
\begin{equation}\label{eq:newton}
u_{n+1}=u_n+\delta_n,\qquad n\ge 0,
\end{equation}
where the update~$\delta_n\in X$ is implicitly given by the {\em linear} equation
\[
\F'(u_n)\delta_n=-\F(u_n),\qquad n\ge 0.
\]
Naturally, for this iteration to be well-defined, we need to assume that~$\F'(u_n)$ is invertible for all~$n\ge 0$, and that~$\{u_n\}_{n\ge 0}\subset\Xi$. 

\subsection{An Adaptive Prediction Strategy}\label{sc:newton}

In order to improve the reliability of the Newton method~\eqref{eq:newton} in the case that the initial guess~$u_0$ is relatively far away from a root $u_{\infty}\in\Xi$ of~$\F$, $\F(u_\infty)=0$, introducing some damping in the Newton method is a well-known remedy. In that case~\eqref{eq:newton} is rewritten as
\begin{equation}\label{eq:damped}
u_{n+1}=u_n-\Delta t_n\F'(u_n)^{-1}\F(u_n),\qquad n\ge 0,
\end{equation} 
where $\Delta t_n>0$, $n\ge 0$, is a damping parameter that may be adjusted {\em adaptively} in each iteration step. 

Provided that~$\F'(u)$ is invertible on a suitable subset of~$\Xi\subset X$, we define the {\em Newton transform} by
\[
u\mapsto\NF(u):=-\F'(u)^{-1}\F(u).
\]
Then, rearranging terms in~\eqref{eq:damped}, we notice that
\begin{equation*}
\frac{u_{n+1}-u_{n}}{\Delta t_n}=\NF(u_n), \qquad n\ge 0,
\end{equation*}
i.e., \eqref{eq:damped} can be seen as the discretization of the {\em Davydenko-type system},
\begin{equation}\label{eq:davy}
\begin{split}
\dot{u}(t)&=\NF(u(t)), \quad t\geq 0,\qquad
 u(0)=u_0,
 \end{split}
\end{equation}
by the forward Euler scheme, with step size~$\Delta t_n>0$. 

For~$t\in[0,\infty)$, the solution~$u(t)$ of~\eqref{eq:davy}, if it exists, defines a trajectory in~$X$ that starts at~$u_0$, and that will potentially converge to a zero of~$\F$ as~$t\to\infty$. Indeed, this can be seen (formally) from the integral form of~\eqref{eq:davy}, that is,
\begin{equation*}
\F(u(t))=\F(u_0)e^{-t},\qquad t\ge 0,
\end{equation*}
which implies that~$\F(u(t))\to 0$ as~$t\to\infty$.

Now taking the view of dynamical systems, our goal is to compute an upper bound for the value of the step sizes~$\Delta t_n>0$ from~\eqref{eq:damped}, $n\ge 0$, so that the discrete forward Euler solution~$\{u_n\}_{n\ge 0}$ from~\eqref{eq:damped} stays reasonably close to the continuous solution of~\eqref{eq:davy}. Specifically, for a prescribed tolerance~$\tau>0$, a Taylor expansion analysis (see~\cite[Section~2]{AmreinWihler:15} for details) reveals that 
\[
u(t)=u_0+t\NF(u_0)+\frac{t^2}{2h}\eta_h+\mathcal{O}(t^3)+\mathcal{O}(t^2h\|\NF(u_0)\|_X^2),
\]
where, for any sufficiently small~$h>0$, we let~$\eta_h=\NF(u_0+h\NF(u_0))-\NF(u_0)$. Hence, after the first time step of length~$\Delta t_0>0$ there holds
\begin{equation}\label{eq:O}
u(\Delta t_0)-u_1=\frac{\Delta t_0^2}{2h}\eta_h+\mathcal{O}(\Delta t_0^3)+\mathcal{O}(\Delta t_0^2h\|\NF(u_0)\|_X^2),
\end{equation}
where~$u_1$ is the forward Euler solution from~\eqref{eq:damped}. Therefore, upon setting
\[
\Delta t_0=\sqrt{2\tau h\|\eta_h\|_X^{-1}},
\] 
we arrive at
\[
\|u(\Delta t_0)-u_1\|_X\le\tau+\mathcal{O}(\Delta t_0^3)+\mathcal{O}(\Delta t_0^2h\|\NF(u_0)\|_X^2).
\]
In order to balance the~$\mathcal{O}$-terms in~\eqref{eq:O} it is sensible to make the choice
\[
h=\mathcal{O}(\Delta t_0\|\NF(u_0)\|_X^{-2}),
\] 
i.e.,
\begin{equation}\label{eq:h}
h=\gamma \Delta t_0\|\NF(u_0)\|_X^{-2},
\end{equation}
for some parameter~$\gamma>0$. This leads to the following \emph{adaptive Newton algorithm}.\\

\begin{algorithm}~\label{al:zs}
Fix a tolerance $\tau>0$ as well as a parameter~$\gamma>0$, and set~$n\gets0$.
\begin{algorithmic}[1]
\State Start the Newton iteration with an initial guess $ u_{0}\in\Xi$.
\If {$n=0$} {choose
\[
\Delta t_0=\min\left\{\sqrt{2\tau\norm{\NF(u_{0})}_{X}^{-1}},1\right\},
\]
based on~\cite[Algorithm~2.1]{AmreinWihler:15},}
\ElsIf {$n\ge 1$} 
\State let $\kappa_{n}=\Delta t_{n-1}$, and $h_n=\gamma \kappa_n\|\NF(u_n)\|_X^{-2}$ based on~\eqref{eq:h};
\State define the Newton step size
\begin{equation}\label{eq:knew}
\Delta t_n=\min\left\{\sqrt{2\tau h_n\norm{\NF(u_n+h_n\NF(u_n))-\NF(u_n)}^{-1}_{X}},1\right\}.
\end{equation}
\EndIf
\State Compute~$u_{n+1}$ based on the Newton iteration~\eqref{eq:damped}, and go to~({\footnotesize 3:}) with $n\leftarrow n+1 $. 
\end{algorithmic}
\end{algorithm}

\begin{remark}
The following comments are noteworthy.
\begin{enumerate}[(i)]
\item We remark that the purpose of the above derivations is to work under minimal structural assumptions on the dynamics caused by the Newton transform. This is particularly important in the context of nonlinear singularly perturbed problems (with~$\e\ll 1$): Indeed, for such problems (as~$\e\to 0$) the inverse of the derivative of the associated nonlinear operator will typically become highly irregular (even on a local level), and, for instance, local or global Lipschitz type or boundedness assumptions (involving, e.g., second derivative information) will in general not be available with practically feasible constants. We remark, however, that more sophisticated step size prediction schemes can be derived if the structure of the nonlinearities can be exploited further; see, in particular, the monograph~\cite{5}.
\item The minimum in~\eqref{eq:knew} ensures that the step size~$\Delta t_n$ is chosen to be~1 whenever possible. Indeed, this is required in order to guarantee quadratic convergence of the Newton iteration close to a root (provided that the root is simple).
\item The preset tolerance~$ \tau $ in the above adaptive strategy will typically be fixed {\em a priori}. Here,  for highly nonlinear problems featuring numerous or even infinitely many solutions, it is typically mandatory to select~$\tau\ll 1$ small in order to remain within the attractor of the given initial guess. This is particularly important if the starting value is relatively far away from a solution.
\eremk
\end{enumerate}
\end{remark}

\section{Newton-Galerkin $hp$-FEM for Semilinear Elliptic Problems}\label{sc:FEM}

The purpose of this section is to derive an $hp$-version Newton-Galerkin finite element formulation for the solution of~\eqref{eq:poisson}. To this end, we will apply the abstract adaptive Newton framework from the previous Section~\ref{sc:anewton}, and apply an $hp$-discretization methodology for the resulting sequence of (locally) linearized problems.

\subsection{Newton Linearization}
In order to apply an adaptive Newton method as introduced in Section~\ref{sc:anewton} to the nonlinear problem~\eqref{eq:F0}, note that the Fr\'echet-derivative of $ \F_{\e} $ from~\eqref{eq:F0} at~$u\in X$ is given by
\begin{equation*}
\left \langle \F_{\e}'(u)w,v\right \rangle = \int_{\Omega}\left\{\e w'v'-f'(u)wv\right\}\dx,\qquad v,w\in X=H^1_0(\Omega),
\end{equation*}
where we write~$f'\equiv\partial_uf$. We note that, if~$f'(u)\in L^{1}(\Omega)$, then $\F_{\e}'(u)$ is a well-defined linear and bounded mapping from~$X$ to~$X'$; see~\cite[Lemma~A.1]{AmreinWihler:15}.

Then, given an initial guess~$u_0\in X$ for~\eqref{eq:F0}, the Newton method~\eqref{eq:damped} is to find~$u_{n+1}\in X$ from~$u_n\in X$, $n\ge 0$, such that
\begin{equation*}
\F_\e'(u_n)(u_{n+1}-u_n)=-\Delta t_n\F_\e(u_n),
\end{equation*}
in~$X'$. 
Equivalently,
\begin{equation}\label{eq:weak}
a_\e(u_n; u_{n+1},v)=a_\e(u_n; u_{n},v)-\Delta t_n\ell_\e(u_n;v)\qquad\forall v\in X,
\end{equation}
where, for {\em fixed}~$u\in X$, 
\begin{equation*}
\begin{aligned}
a_{\e}(u;w,v)&:=\int_{\Omega}\left\{\e w'v'-f'(u)wv\right\}\dx ,\\
l_{\e}(u;v)&:=\int_{\Omega}\left\{\e u'v'-f(u)v\right\}\dx
\end{aligned}
\end{equation*}
are bilinear and linear forms on~$X\times X$ and~$X$, respectively.

\begin{remark}
Incidentally, if there are constants~$\underline{\lambda},\overline{\lambda}\ge 0$ with $\e C_P^{-2}>\overline{\lambda}$ such that $ -\underline\lambda \leq f'(u) \leq \overline{\lambda}$ holds for all $u\in\mathbb{R} $, where $C_P:=\nicefrac{(b-a)}{\pi}>0$ is the (best) constant in the Poincar\'e inequality on~$\Omega=(a,b)$, i.e.,\begin{equation}\label{eq:P}
\|w\|_{0}\le C_P\|w'\|_{0}\qquad\forall w \in X,
\end{equation}
see~\cite[Theorem~257]{HLP52}, then, it is elementary to show that the formulation~\eqref{eq:weak} is a linear second-order diffusion-reaction problem, which is coercive on~$X$. In consequence, for any given $ u_n \in X $, it has a unique solution~$u_{n+1}\in X$; see~\cite{AmreinDiss15} for details.
\eremk
\end{remark}

\subsection{Linearized Finite Element Approach}
In order to provide a numerical approximation of~\eqref{eq:poisson}, we will discretize the weak formulation~\eqref{eq:weak} by means of an $hp$-finite element method, which, in combination with the Newton iteration, constitutes a Newton-Galerkin approximation scheme. Furthermore, in view of deriving an $hp$-adaptive algorithm, we will develop an $\e$-robust $hp$-version {\em a posteriori} residual analysis for the linear finite element formulation.

\subsubsection{$hp$-Spaces}
We introduce a partition~$\T=\{K_j\}_{j=1}^N$ of~$N\ge 1$ (open)
elements~$K_j=(x_{j-1},x_j)$, $j=1,2,\ldots,N$, on~$\Omega=(a,b)$, with $a=x_0<x_1<x_2<\ldots<x_{N-1}<x_N=b$. The length of an element~$K_j$ is denoted by~$h_j=x_j-x_{j-1}$,
$j=1,2,\ldots,N$. For each element $K_j \in \T$, it will be convenient 
to introduce the patch $\widetilde K_j = \bigcup \{K'\in\T\,|\, 
\overline{K'}\cap \overline{K_j} \ne \emptyset\}$ as the union of~$K_j$ and
of the elements adjacent to it. 
In addition, to each element~$K_j$ we associate a
polynomial degree~$p_j\ge 1$, $j=1,2,\ldots,N$. These numbers are
stored in a polynomial degree vector~$\p=(p_1,p_2,\ldots,p_N)$. Then,
we define an $hp$-finite element space by
\[
\V=\left\{v\in H^1_0(\Omega):\,
v|_{K_j}\in\mathbb{P}_{p_j}(K_j),\,j=1,2,\ldots,N\right\},
\]
where, for~$p\ge 1$, we denote by~$\mathbb{P}_p$ the space of all
polynomials of degree at most~$p$. We say that the pair $(\T,\p)$ 
of a partition $\T$ and of a degree vector $\p$ is $\mu$-shape regular, for some constant~$\mu>0$ independent of~$j$, if 
\begin{equation}
\label{eq:gamma-shape-regular}
\mu^{-1} h_{j+1} \leq h_{j} \leq \mu h_{j+1}, 
\qquad  
\mu^{-1} p_{j+1} \leq p_{j} \leq \mu p_{j+1}, 
\qquad j=1,\ldots,N-1, 
\end{equation}
i.e., if both the element sizes and polynomial degrees of neighboring 
elements are comparable. 

\subsubsection{Linear $hp$-FEM}
We define the numerical approximation of~\eqref{eq:poisson} in an iterative way as follows: For~$n=0$, we let~$\V$ be an initial $hp$-discretization space, and $\u_{0} \in \V$ a suitable initial guess for the Newton iteration. Furthermore, for~$n\ge 0$, and given~$\u_n \in \V$, we find a finite element approximation~$\u_{n+1}\in \V$  of~\eqref{eq:weak} as the solution of the weak formulation
\begin{equation}\label{eq:fem}
a_\e(\u_n; \u_{n+1},v)=a_\e(\u_n; \u_n,v)-\Delta t\ell_\e(\u_n;v)\qquad\forall v\in \V.
\end{equation}
Here, $\Delta t$ takes the role of a (possibly varying) step size parameter in the adaptive Newton scheme as described earlier in Section~\ref{sc:newton}.
 
\section{{\em A Posteriori} Residual Analysis}

Recalling~\eqref{eq:fem}, for given~$\Delta t>0$, let us introduce the function
\begin{equation}\label{eq:ut}
u_{n+1}^{(\Delta t,\hp)}:=\u_{n+1}-(1-\Delta t)\u_n,
\end{equation}
as well as
\begin{equation}\label{eq:ft}
f^{\Delta t}(\u_{n+1}):=\Delta tf(\u_n)+f'(\u_n)(u_{n+1}^\hp-\u_n).
\end{equation}
Then, rearranging terms, we can rewrite~\eqref{eq:fem} in the following form:
\begin{equation}
 \label{start}
 \begin{aligned}
\e\int_{\Omega}{ (u_{n+1}^{(\Delta t,\hp)})'v'} \dx& = \int_{\Omega}{ f^{\Delta t} (u_{n+1}^\hp) v }\dx \qquad \forall v  \in V_{0}^h.
\end{aligned}
\end{equation}

The aim of this section is to derive {\em a posteriori} residual-based bounds for~\eqref{start}. In order to measure the discrepancy between the finite element discretization~\eqref{eq:fem} and the original problem~\eqref{eq:poisson}, an obvious quantity to bound is the residual~$\F_{\e}(\u_{n+1})$ in~$X'$. In order to proceed in this direction, we notice that the adaptively chosen damping parameter~$\Delta t$ in the Newton method~\eqref{eq:fem} will equal~1 sufficiently close to a root of~$\F_\e$. For this reason, as was proposed in~\cite{AmreinWihler:15}, we may focus on the `shifted' residual~$\F_{\e}(u_{n+1}^{(\Delta t,\hp)})$ in~$X'$ instead; indeed, estimating this quantity turns out to be very natural in view of the linearized finite element discretization~\eqref{start}.

\subsection{Robust Interpolation Estimates}\label{sc:int}
For the purpose of deriving an (upper) \emph{a posteriori} residual estimate that is robust with respect to the singular perturbation parameter~$\e$ (for~$0<\e\ll 1$) as well as optimally scaled with respect to the local element sizes~$h_j$ and polynomial degrees~$p_j$, it is crucial to construct an interpolation operator that is \emph{simultaneously} $L^2$- and $H^1$-stable. This will be accomplished in the current section; see Proposition~\ref{pr:intp} and Corollary~\ref{cr:int} below.

\begin{proposition}\label{pr:intp} 
Let the pair $(\T,\p)$ be $\mu$-shape regular; see \eqref{eq:gamma-shape-regular}. Then, there exists an interpolation operator~$\intp:\,H^1_0(\Omega)\to\V$ such that, for any~$j=1,2,\ldots,N$, and for all~$v\in H^1_0(\Omega)$, there holds
\begin{equation}\label{eq:CI}
\begin{split}
\NN{v-\intp v}_{L^2(K_j)}&\le C_I\min\left\{\NN{v}_{L^2(\widetilde K_j)}, h_jp_j^{-1}\NN{v'}_{L^2(\widetilde K_j)}\right\},\\
\NN{(v-\intp v)'}_{L^2(K_j)}&\le C_I\NN{v'}_{L^2(\widetilde K_j)}.
\end{split}
\end{equation}
%Furthermore, for~$i=1,\ldots,N-1$, we have the nodal estimates 
%\begin{equation*}
%\begin{split}
%|(v &- \intp v)(x_i)|^2\\ 
%& \leq C_I \Bigg( \frac{1} {h_i + h_{i+1}} \|v - \intp v\|_{L^2(K_i \cup K_{i+1})}^2 \\ 
%& \quad\qquad+  
%  \|v - \intp v\|_{L^2(K_i \cup K_{i+1})}^2 
%  \|(v - \intp v)^\prime\|_{L^2(K_i \cup K_{i+1})}^2 \Bigg). 
%\end{split}
%\end{equation*}
Here, $C_I>0$ is a constant that depends solely on $\mu$; in particular, 
it is independent of~$v$, $\mathcal{T}$, and of~$\bm p$. 
\end{proposition}

\begin{proof} 
Let us, without loss of generality, assume that~$\Omega=(0,1)$. The result can be shown with 
the techniques employed in the higher-dimensional case in \cite[Cor.~{3.7}]{karkulik-melenk15}. In the present, 
one-dimensional case, a significantly simpler argument can be brought to bear, which allows us to 
make the paper self-contained. Let $x_{-1} = -h_1$ and 
$x_{N+1} = 1+h_N$ and $\varphi_i$, $i=0,\ldots,N+1$ be the 
standard piecewise linear hat functions associated with the nodes 
$x_i$, $i=-1,\ldots,N+1$. The extra nodes $x_{-1}$ and $x_{N+1}$ define 
in a natural way the elements $K_{0}$ and $K_{N+1}$. 
The (open) patches $\omega_i$, $i=0,\ldots,N$, are given by the supports of the 
functions $\varphi_i$, i.e., $\omega_i = (\operatorname*{supp} \varphi_i)^\circ
 = K_i \cup K_{i+1} \cup \{x_i\}$. 

Polynomial approximation (see, e.g., \cite[Proposition~{A.2}]{MelenkCLEM}) gives
the existence of an interpolation operator 
$J_p:L^2(-1,1) \rightarrow {\mathbb P}_p(-1,1)$ that is uniformly 
(in $p\ge 0$)
stable, i.e., $\|J_p v\|_{L^2(-1,1)} \leq C \|v\|_{L^2(-1,1)}$ for all 
$v \in L^2(-1,1)$ and has the following properties for $v \in H^1(-1,1)$: 
\[
(p+1) \|v - J_p v\|_{L^2(-1,1)} + 
\|(v - J_p v)^\prime\|_{L^2(-1,1)} \leq  C \|v^\prime\|_{L^2(-1,1)}.  
\]
Furthermore, if $v$ is antisymmetric with respect to the midpoint $x = 0$, 
then $J_p v$ can be assumed to be antisymmetric as well, i.e., 
$(J_p v)(0) = 0$ (this follows from studying the antisymmetric part 
of the original function $J_p v$). 

The approximation $\intp v$ is now constructed with 
the aid of a ``partition of unity argument'' as described in 
\cite[Theorem~{2.1}]{babuska-melenk96}. For $\omega_0$ and $\omega_N$, 
extend $v$ anti-symmetrically, i.e., $v(x):=-v(-x)$ for $x \in K_{0}$
and $v(x):=-v(1-x)$ for $x \in K_{N+1}$. Then $v$ is defined on each
patch $\omega_i$, $i=0,\ldots,N$. For each patch $\omega_i$, 
let $p^\prime_i:= \min\{p_{i},p_{i+1}\}$ (with the understanding 
$p_0 = p_1$ and $p_{N+1} = p_N$). The above operator $J_p$ then induces
for each patch $\omega_i$ by scaling an operator 
$J^i:L^2(\omega_i) \rightarrow {\mathcal P}_{p^\prime_i-1}(\omega_i)$ 
with the following properties:  
\[
\frac{p_i^\prime+1}{h_i}\|v - J^i v\|_{L^2(\omega_i)} + 
\|(v - J^i v)^\prime\|_{L^2(\omega_i)} \leq  C \|v^\prime\|_{L^2(\omega_i)};  
\]
here, we have exploited the $\mu$-shape regularity of the mesh. 
We note that $(J^0 v)(0) = 0$ and $(J^N v)(1) = 0$. Also, the operators 
$J^i$ are uniformly (in the polynomial degree) stable in $L^2(\omega_i)$. 
The approximation $\intp v$ is now taken to be 
$\intp v := \sum_{i=0}^N \varphi_i J^i v$. The desired approximation 
properties follow now from \cite[Theorem~{2.1}]{babuska-melenk96}. 
%
%Finally, the nodal estimate results from the observation 
%that at the mesh nodes, there holds the identity $\intp v(x_i) = (J^i v)(x_i)$, and from a multiplicative trace estimate (see~\cite[Lemma~A.1]{MW15}). 
\end{proof}

The above proposition implies the following bounds.

\begin{corollary}\label{cr:int}
For~$v\in H^1_0(\Omega)$, the interpolant from Proposition~\ref{pr:intp} satisfies
\begin{equation*}
\NN{v-\intp v}^2_{L^2(K_j)}\le
C^2_I\alpha_j\NNN{v}_{\e,\widetilde K_j}^2 ,
\qquad j=1,2,\ldots,N,
\end{equation*}
and
\begin{equation*}
\big|(v-\intp v)(x_j)\big|^2\le
C_I^2\gamma_j\left(\NNN{v}_{\e,\widetilde K_{j}}^2+\NNN{v}^2_{\e,\widetilde K_{j+1}}\right),
\qquad j=1,2,\ldots,N-1.
\end{equation*}
Here, we let
\begin{align}\label{eq:alpha}
\alpha_j&=
\min\left\{\e^{-1}h_j^2p_j^{-2},1\right\},\qquad j=1,\ldots, N,
\end{align}
and, with
\begin{align}
\beta_j&=\alpha_jh_j^{-1}+2\sqrt{\e^{-1}\alpha_j},\label{eq:beta}
\end{align}
we define
\begin{equation}\label{eq:gamma}
\gamma_j=\frac{\beta_j\beta_{j+1}}{\beta_j+\beta_{j+1}},
\end{equation}
for~$1\le j\le N-1$, and~$\gamma_0=\gamma_N=0$. 
%The constant~$C>0$ is
%independent of~$u$, $\u$, $f$, $\e$, $\T$, and of~$\p$.
Moreover, $C_I$ is the constant from~\eqref{eq:CI}.
\end{corollary}

\begin{proof}
We proceed along the lines of~\cite{Ve1998}. Using the bounds from
Proposition~\ref{pr:intp}, for each element~$K_j\in\T$, we have that
\begin{align*}
\NN{v-\intp v}^2_{L^2(K_j)}&\le 
C^2_I\frac{h_j^2}{\e p_j^2}\e\NN{v'}^2_{L^2(\widetilde K_j)}.
\end{align*}
Furthermore, 
\begin{align*}
\NN{v-\intp v}^2_{L^2(K_j)}
&\le C^2_I\NN{v}_{L^2(\widetilde K_j)}^2.
\end{align*}
Combining these two estimates, yields the first bound.
 
In order to prove the second estimate, we apply, for~$1\le j\le N-1$,
a multiplicative trace inequality; see~\cite[Lemma~A.1]{MelenkWihler:15}):
\begin{align*}
\left|(v-\intp v)(x_j)\right|^2
&\le h_j^{-1}\NN{v-\intp v}^2_{L^2(K_j)}\\
&\quad+2\NN{v-\intp
  v}_{L^2(K_j)}\NN{(v-\intp v)'}_{L^2(K_j)}.
\end{align*}
Then, invoking the above bounds as well as the estimates from
Proposition~\ref{pr:intp}, we get 
\begin{align*}
\big|(v-\intp v)(x_j)\big|^2
&\le C_I^2\left(\alpha_jh_j^{-1}\NNN{v}_{\e,\widetilde K_j}^2
+2\sqrt{\alpha_j}\NNN{v}_{\e,\widetilde K_j}\NN{v'}_{L^2(\widetilde K_j)}\right)\\
&\le C_I^2\left(\alpha_jh_j^{-1}\NNN{v}_{\e,\widetilde K_j}^2
+2\sqrt{\e^{-1}\alpha_j}\NNN{v}_{\e,\widetilde K_j}^2\right)\\
&\le C_I^2\beta_j\NNN{v}_{\e,\widetilde K_j}^2,
\end{align*}
with $\beta_j$ from~\eqref{eq:beta}. Since~$x_j$ is also a boundary
point of~$K_{j+1}$, we similarly obtain that
\[
\big|(v-\intp v)(x_j)\big|^2
\le C_I^2\beta_{j+1}\NNN{v}_{\e,\widetilde K_{j+1}}^2.
\]
Therefore,
\begin{align*}
\big|(v-\intp v)(x_j)\big|^2
&=\frac{\beta_{j+1}}{\beta_{j}+\beta_{j+1}}\big|(v-\intp
v)(x_j)\big|^2\\
&\quad +\frac{\beta_{j}}{\beta_{j}+\beta_{j+1}}\big|(v-\intp
v)(x_j)\big|^2\\ &\le
C_I^2\gamma_j\left(\NNN{v}_{\e,\widetilde K_{j}}^2+\NNN{v}^2_{\e,\widetilde K_{j+1}}\right),
\end{align*}
with~$\gamma_j$ from~\eqref{eq:gamma}. Thus, we have shown the second
estimate.
\end{proof}

\subsection{Robust \emph{A Posteriori} Residual Estimate}

Based on the previous interpolation analysis in Section~\ref{sc:int}, we are now in a position to prove the following main result.

\begin{theorem}\label{thm:main}
\label{thm:1}
For~$u_{n+1}^{(\Delta t,\hp)}$ from~\eqref{eq:ut} there holds the upper \emph{a posteriori} bound
\begin{equation}
\label{eq:upperbound}
\NNN{\F_\e(u_{n+1}^{(\Delta t,\hp)})}_{X',\e}^2
\preccurlyeq \delta_{n,\Omega}^2+\sum_{j=1}^N{\eta_{n,j}^{2}},
\end{equation}
with
\begin{equation}
\label{Femerror}
\begin{split}
\eta_{n,j}^2:&= \alpha_{j} \norm{f^{\Delta t}(\u_{n+1})+\e (u_{n+1}^{(\Delta t,\hp)})''}_{0,K_j}^2\\
&\quad+\frac{1}{2}\e^2\gamma_j\left|\jmp{(u_{n+1}^{(\Delta t,\hp)})'}(x_j)\right|^2
+\frac{1}{2}\e^2\gamma_{j-1}\left|\jmp{(u_{n+1}^{(\Delta t,\hp)})'}(x_{j-1})\right|^2,
\end{split}
\end{equation}
for~$1\le j\le N$, and
\begin{equation}
\label{Newtonerror}
\delta_{n,\Omega}:=\norm{f^{\Delta t}(u_{n+1}^{\hp})-f(u_{n+1}^{(\Delta t,\hp)})}_{0,\Omega}.
\end{equation}
Here, $\alpha_j$ and~$\gamma_j$ are defined in~\eqref{eq:alpha} and~\eqref{eq:gamma}, respectively, and we set~$\gamma_0=\gamma_N=0$.
\end{theorem}

\begin{remark}
We emphasize that the constants~$\alpha_j$ and~$\e^2\alpha_j\gamma_j$ appearing in the error indicators~$\eta_{K_j}$ from~\eqref{Femerror} remain bounded as~$h_j,\e\to 0$ (and~$p_j\to\infty$). In addition, we note that $\nicefrac{1}{2} \min\{\beta_j,\beta_{j+1}\} \leq \gamma_j \leq \min\{\beta_j,\beta_{j+1}\}$, and that 
$2\sqrt{\e^{-1} \alpha_j } \leq \beta_j \leq 3 \sqrt{\e^{-1} \alpha_j}$.
\eremk
\end{remark}

\begin{remark}
For \emph{linear problems}, the (adaptive) Newton linearization is redundant. In particular, upon setting~$\Delta t=1$, we infer that~$\delta_{n,\Omega}=0$ in~\eqref{eq:upperbound}, and the residual is bounded by the local error indicators~$\eta_{n,j}$ only. These quantities, in turn, can be exploited for the purpose of designing an $hp$-adaptive mesh refinement strategy, as outlined, for example, in Section~\ref{sc:hpadapt} below. We emphasize that this approach has been studied earlier in the report~\cite{MelenkWihler:15}, where a number of numerical tests in the context of $hp$-adaptive FEM for singularly perturbed linear problems have been performed; in particular,  these experiments have underlined the robustness of the bound~\eqref{eq:upperbound} as~$\e\to 0$ in the linear case. In Section~\ref{sc:adaptivity}, we will extend this idea by proposing a suitable interplay between $hp$-mesh refinements and the (adaptive) Newton linearization from Algorithm~\ref{al:zs}.
\eremk
\end{remark}

\begin{proof}[Proof of Theorem~\ref{thm:main}]
Given~$u_{n+1}^{(\Delta t,\hp)}$ and~$f^{\Delta t}(\u_{n+1})$ from~\eqref{eq:ut} and~\eqref{eq:ft}, respectively, and any $v\in H^1_0(\Omega)$, there holds the identity
\begin{equation}\label{eq:Fres}
\left \langle \F_{\e}(u_{n+1}^{(\Delta t,\hp)}),v \right \rangle
=\sum_{j=1}^{N-1}a_j+\sum_{j=1}^N(b_j+c_j),
\end{equation}
where
\begin{alignat}{2}
a_j&:= \e\jmp{(u_{n+1}^{(\Delta t,\hp)})'}(x_j)(v-\intp v)(x_j),&\qquad&1\le j\le N-1,\nonumber\\
b_j&:=\int_{K_j}{\left\{f^{\Delta t}(\u_{n+1})+\e (u_{n+1}^{(\Delta t,\hp)})''\right\}(\intp v-v)}\dx,&&1\le j\le N,\label{bounds}\\
c_j&:=\int_{K_j}{\left\{f^{\Delta t}(\u_{n+1})-f(u_{n+1}^{(\Delta t,\hp)})\right\}v}\dx,&&1\le j\le N;\nonumber
\end{alignat}
see \cite[Proposition~4.3]{AmreinWihler:15}.

We estimate the terms~$a_j$, $b_j$, and $c_j$ individually. First let~$1\le j\le N-1$. Then, $a_j$ from~\eqref{bounds} can be estimated using Corollary~\ref{cr:int} as follows:
\begin{align*}
|a_j|&\le \e\left|\jmp{(u_{n+1}^{(\Delta t,\hp)})'}(x_j)\right||(v-\intp v)(x_j)|\\
&\preccurlyeq\gamma_j^{\nicefrac12}\e\left|\jmp{(u_{n+1}^{(\Delta t,\hp)})'}(x_j)\right|\left(\NNN{v}^2_{\e,\widetilde K_j}+\NNN{v}^2_{\e,\widetilde{K}_{j+1}}\right)^{\nicefrac12}.
\end{align*}
Applying the Cauchy-Schwarz inequality leads to
\begin{align*}
\left|\sum_{j=1}^{N-1}a_j\right|
&\preccurlyeq\left(\sum_{j=1}^{N-1}{\e^2\gamma_j\left|\jmp{(u_{n+1}^{(\Delta t,\hp)})'}(x_j)\right|^2}\right)^{\nicefrac12}
\left(\sum_{j=1}^{N-1}\left(\NNN{v}^2_{\e,\widetilde K_j}+\NNN{v}^2_{\e,\widetilde{K}_{j+1}}\right)\right)^{\nicefrac12}\\
&\preccurlyeq\left(\sum_{j=1}^{N-1}{\e^2\gamma_j\left|\jmp{(u_{n+1}^{(\Delta t,\hp)})'}(x_j)\right|^2}\right)^{\nicefrac12}\NNN{v}_{\e}.
\end{align*}
Furthermore, again using Corollary~\ref{cr:int}, we see that
\begin{align*}
\left|\sum_{j=1}^Nb_j\right|
&\preccurlyeq\sum_{j=0}^N\alpha_j^{\nicefrac12}\norm{f^{\Delta t}(\u_{n+1})+\e (u_{n+1}^{(\Delta t,\hp)})''}_{0,K_j}\NNN{v}_{\e,\widetilde K_j}\\
&\preccurlyeq\left(\sum_{j=1}^N\alpha_j\norm{f^{\Delta t}(\u_{n+1})+\e (u_{n+1}^{(\Delta t,\hp)})''}_{0,K_j}^2\right)^{\nicefrac12}\left(\sum_{j=1}^N\NNN{v}_{\e,\widetilde{K}_j}^2\right)^{\nicefrac12}\\
&\preccurlyeq\left(\sum_{j=1}^N\alpha_j\norm{f^{\Delta t}(\u_{n+1})+\e (u_{n+1}^{(\Delta t,\hp)})''}_{0,K_j}^2\right)^{\nicefrac12}\NNN{v}_{\e}.
\end{align*}
Similarly, there holds
\begin{align*}
\left|\sum_{j=1}^Nc_j\right|&
\preccurlyeq\sum_{j=1}^N \norm{f^{\Delta t}(\u_{n+1})-f(u_{n+1}^{(\Delta t,\hp)})}_{0,K_j}\norm{v}_{0,K_j}\\
&\preccurlyeq\left(\sum_{j=1}^N\norm{f^{\Delta t}(\u_{n+1})-f(u_{n+1}^{(\Delta t,\hp)})}_{0,K_j}^2\right)^{\nicefrac12}\NNN{v}_{\e}.
\end{align*}
Now, applying the Cauchy-Schwarz inequality to~\eqref{eq:Fres} we see that
\[
\begin{aligned}
\abs{\left \langle \F_{\e}(u_{n+1}^{(\Delta t,\hp)}),v \right \rangle}
&\preccurlyeq\left(\delta_{n,\Omega}^2+\sum_{j=1}^N{\eta_{n,j}^{2}}\right)^{\nicefrac{1}{2}}\NNN{v}_{\e}.
\end{aligned}
\]
Dividing by~$\NNN{v}_{\e}$, and taking the supremum for all~$v\in H^1_0(\Omega)$, completes the proof.
\end{proof}

\begin{remark}
Suppose that there exist constants~$\lambda>-C_P^{-2}\e$ and~$L>0$, with~$C_P$ being the Poincar\'e constant from~\eqref{eq:P}, such that
\begin{align*}
(f(x)-f(y))(x-y)\le-\lambda(x-y)^2,\qquad
|f(x)-f(y)|\le L|x-y|,
\end{align*}
for all~$x,y\in\mathbb{R}$. Then, the residual norm $ \NNN{\F_\e(u_{n+1}^{(\Delta t,h)})}_{X',\e}$ and the energy norm $\NNN{u-u_{n+1}^{(\Delta t,h)}}_{\e}$ of the error are equivalent; see~\cite[Theorem~4.5]{AmreinWihler:15} for details.
\eremk
\end{remark}

\begin{remark}
Following along the lines of~\cite[\S4.4.2]{AmreinWihler:15} and~\cite[Appendix]{MelenkWihler:15}, it is possible to prove \emph{$\e$-robust local lower residual bounds} in terms of the error indicators~$\eta_{K_j}$ and the data oscillation terms. This approach, however, results in local efficiency bounds that will be slightly suboptimal with respect to the local polynomial degrees due to the need of applying $p$-dependent norm equivalences (involving cut-off functions).
\eremk
\end{remark}

\section{An $hp$-Adaptive Newton-Galerkin Procedure}\label{sc:adaptivity}

In this section, we will discuss how the \emph{a posteriori} bound from Theorem~\ref{thm:main} can be exploited for the purpose of designing an $hp$-adaptive Newton-Galerkin algorithm for the numerical solution of~\eqref{eq:poisson}. We will begin by revisiting an $hp$-adaptive testing strategy from~\cite{FaWiWi14} (see also~\cite{Wi11,Wi11_2}).

\subsection{$hp$-Adaptive Refinement Procedure}\label{sc:hpadapt}
In order to $hp$-refine the $hp$-finite element space $\V$, we shall apply an $hp$-adaptive algorithm which is based on the following two ingredients.

\subsubsection*{(a) Element marking:} The elementwise error indicators~$\eta_{n,j}$ from~Theorem~\ref{thm:main} are employed in order to mark elements for refinement. More precisely, we fix a parameter~$\vartheta\in(0,1)$, and select elements to be refined according to the \emph{D\"orfler marking} criterion:
\begin{equation}\label{eq:dorf}\tag{D}
\vartheta\sum_{j=1}^N\eta_{n,j}^2\le \sum_{j'=1}^M\eta_{n,j'}^2.
\end{equation}
Here, the indices~$j'$ are chosen such that the error indicators~$\eta_{j'}$ from~\eqref{Femerror} are sorted in descending order, and~$M$ is minimal.

\subsubsection*{(b) $hp$-refinement criterion:} The decision of whether a marked element in step~(a) is refined with respect to~$h$ (element bisection) or~$p$ (increasing the local polynomial order by~1) is based on a smoothness testing approach. Specifically, if the (numerical) solution is considered smooth on a marked element~$K_j$, then the polynomial degree is increased by~1 on that particular element (no element bisection), otherwise the element is bisected (retaining the current polynomial degree~$p_j$ on both subelements). In order to evaluate the smoothness of the solution~$\uu$ on a marked element~$K_j$, we employ an elementwise smoothness indicator as introduced in~\cite[Eq.~(3)]{FaWiWi14},
\begin{equation}\label{eq:hp}\tag{F}
\mathcal{F}^{p_j}_j\left[\uu\right]:=
\frac{\NN{\frac{\dd^{p_j-1}}{\dd x^{p_j-1}}\uu}_{L^\infty(K_j)}}{h_j^{-\nicefrac12}\NN{\frac{\dd^{p_j-1} \uu}{\dd x^{p_j-1}}}_{L^2(K_j)}+\frac{1}{\sqrt2}h_j^{\nicefrac12}\NN{\frac{\dd^{p_j} \uu}{\dd x^{p_j}}}_{L^2(K_j)}},
\end{equation}
if~$\frac{\dd^{p_j-1}}{\dd x^{p-1}}\uu|_{K_j}\not\equiv 0$, and~$\mathcal{F}^{p_j}_j[\uu]=1$ otherwise. Here, the basic idea is to consider the continuous Sobolev embedding~$H^1(K_j)\hookrightarrow L^\infty(K_j)$, which implies that
\[
\sup_{v\in H^1(K_j)}\frac{\NN{v}_{L^\infty(K_j)}}{h_j^{-\nicefrac12}\NN{v}_{L^2(K_j)}+\frac{1}{\sqrt2}h_j^{\nicefrac12}\NN{v'}_{L^2(K_j)}}\le 1;
\]
see~\cite[Proposition~1]{FaWiWi14}. In particular, it follows that~$\mathcal{F}_j^{p_j}[\uu]\le 1$. For ease of evaluation, note that, by taking the derivative of order~$p_j-1$ in the definition~\eqref{eq:hp}, the smoothness indicator~$\mathcal{F}^{p_j}_j[\uu]$ is evaluated for linear functions only; in this case, it can be shown that
\[
\frac12\approx\frac{\sqrt{3}}{\sqrt{6}+1}\le\mathcal{F}_j^{p_j}\left[\uu\right]\le 1;
\] 
cf.~\cite[Section~2.2]{FaWiWi14}. For a given smoothness testing parameter~$\zeta\in(\nicefrac{\sqrt{3}}{(\sqrt{6}+1)},1)$, the numerical solution~$\uu$ is classified smooth on~$K_j$ if
\begin{equation}\label{eq:st}
\mathcal{F}^{p_j}_j\left[\uu\right]\ge\zeta,
\end{equation} 
and otherwise nonsmooth. Incidentally, representing the local solution~$\uu|_{K_j}$ in terms of (local) Legendre polynomials (or more general Jacobi polynomials), any derivatives of~$\uu$ can be evaluated exactly by means of appropriate recurrence relations. For further information, we refer to the paper~\cite{FaWiWi14} (see also~\cite{Wi11_2,Wi11}), where the smoothness testing strategy described above has been introduced and discussed in detail.

\subsection{Fully Adaptive Newton-Galerkin Method}
We will now propose a procedure that provides an interplay of the adaptive Newton methods presented in~Section~\ref{sc:anewton} and automatic $hp$-finite element mesh refinements based on the {\em a posteriori} error estimate from Theorem~\ref{thm:1} (as outlined in the previous Section~\ref{sc:hpadapt}). To this end, we make the assumption that the Newton sequence $\left\{u_{n+1}^{(\Delta t_n,\hp)}\right\}_{n\ge 0} $ given by~\eqref{eq:fem} and~\eqref{eq:ut}, with a (possibly varying) step size~$\Delta t_n$, is well-defined as long as the iterations are being performed. 

\begin{algorithm}\label{al:full}
Given a (coarse) starting mesh $\mathcal{T}$ in~$\Omega$, with an associated (low-order) polynomial degree distribution~$\bm p$, and an initial guess $ u_{0}^{\hp} \in  \V $. Set~$n\gets 0$.
\begin{algorithmic}[1]
\State
Determine the Newton step size parameter $\Delta t_n$ based on~$\u_n$ by the adaptive procedure from Algorithm~\ref{al:zs}. 
\State 
Compute the FEM solution~$u_{n+1}^{\hp}$ from~\eqref{eq:fem} with step size~$\Delta t_n$ on the space~$\V$. Furthermore, obtain~$ u_{n+1}^{(\Delta t_n,\hp)} $ in~\eqref{eq:ut}, and evaluate the corresponding error indicators $ \{\eta_{n,j}\}_j $, and $\delta_{n,\Omega} $ from~\eqref{Femerror} and~\eqref{Newtonerror}, respectively.  
\If {
\begin{equation}\label{eq:test}
\delta_{n,\Omega}^2\le \theta\sum_{j}{\eta_{n,j}^2}
\end{equation}
holds,}
\myState {$hp$-refine the space $\V$ adaptively based on the marking criterion~\eqref{eq:dorf}, and using the $hp$-strategy from Section~\ref{sc:hpadapt}; repeat step~({\footnotesize\sc 2:}) with the new space~$\V$.}
\Else
\myState{i.e., if~\eqref{eq:test} is not fulfilled, then set~$n\leftarrow n+1$, and perform another adaptive Newton step by going back to~({\footnotesize\sc 1:}).}
\EndIf
\end{algorithmic}
\end{algorithm}

\subsection{Numerical Experiments}

Let us provide some comments on how Algorithm~\ref{al:full} is implemented in the ensuing examples.

\begin{itemize}
\item The Newton transform $\NF(\u_n)$ required for the computation of the step size parameter~$\Delta t_n$ in step~({\footnotesize1:}) is approximated using the $hp$-FEM on the current mesh.
\item If the criterion~\eqref{eq:test} in step~({\footnotesize3:}) is satisfied, the mesh is refined, and then step~({\footnotesize2:}) is repeated based on the previously computed solution~$u_{n+1}^{\hp}$ as interpolated on the refined mesh.
\item The linear systems resulting from the finite element discretization~\eqref{start} are solved by means of a direct solver; in this way, this approach differs from inexact Newton methods as discussed, e.g., in~\cite{ErVo13}; see also~\cite{5}.
\item The parameters are chosen to be $\gamma=0.5$ and $\tau=0.1$ in Algorithm~\ref{al:zs}, $\vartheta=0.5$ in~\eqref{eq:dorf}, $\zeta=0.6$ in~\eqref{eq:st}, and~$\theta=0.5$ in~\eqref{eq:test}. Unless stated differently, the procedure is initiated with a uniform initial mesh~$\T$ consisting of 10 elements, and a polynomial degree distribution~$\bm p=(1,\ldots 1)$.
\end{itemize}

%%%%%%%%%%%%%%%%%%%%%%%%%%%%%%%%%%%%%%%%%%%%%

\begin{example}\label{ex:slin0}
We begin by looking at the (non-singularly perturbed) \emph{Bratu} equation given by
\begin{equation}\label{eq:Bratu}
-u''=\exp(u+1)\quad\text{on }\Omega=(0,1),\qquad u(0)=u(1)=0.
\end{equation}
This problem features two positive and concave solutions, which we aim to find by means of the fully adaptive Newton-Galerkin procedure introduced above (cf.~also~\cite[Example~3.4]{AmreinWihler:14}). To this end, we let the initial guesses to be $\u_0\equiv0$ and~$\u_0(x)=\pi_1(10x(1-x))$, where~$\pi_1$ denotes the piecewise linear interpolation operator on the initial mesh. For these choices, Algorithm~\ref{al:full} generates the two solutions in bold and dashed lines, respectively, in Figure~\ref{fig:slin0} (left). Since both solutions are smooth, the $hp$-adaptive testing strategy~\eqref{eq:st} dictates pure $p$-refinements in all steps as expected, and, consequently, the residuals decay at an exponential rate with respect to the number of degrees of freedom; see Figure~\ref{fig:slin0} (right).
\eremk
\end{example}

\begin{figure}
\begin{center}
	\includegraphics[width=0.45\linewidth]{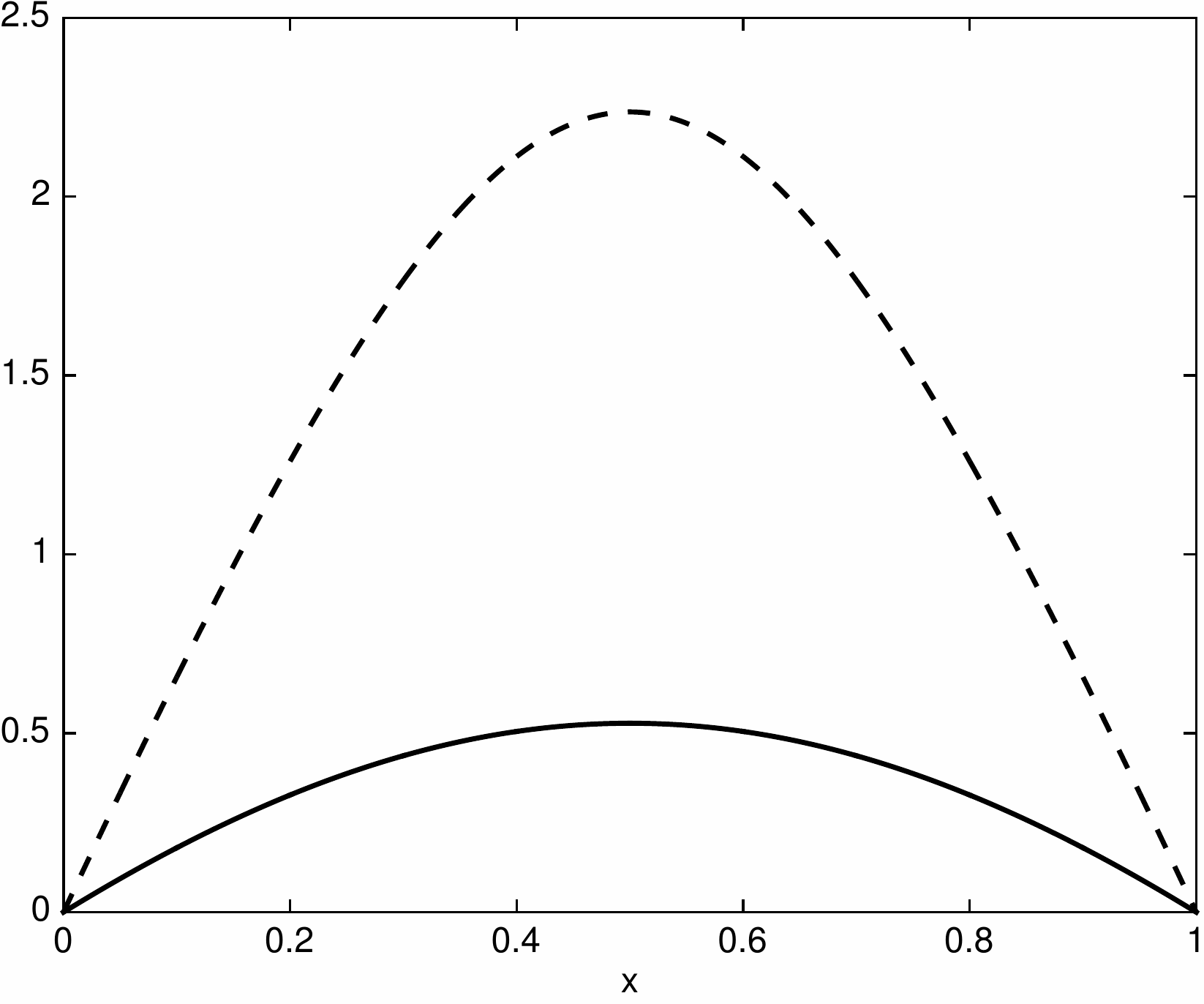}\hfill
	\includegraphics[width=0.45\linewidth,height=0.375\linewidth]{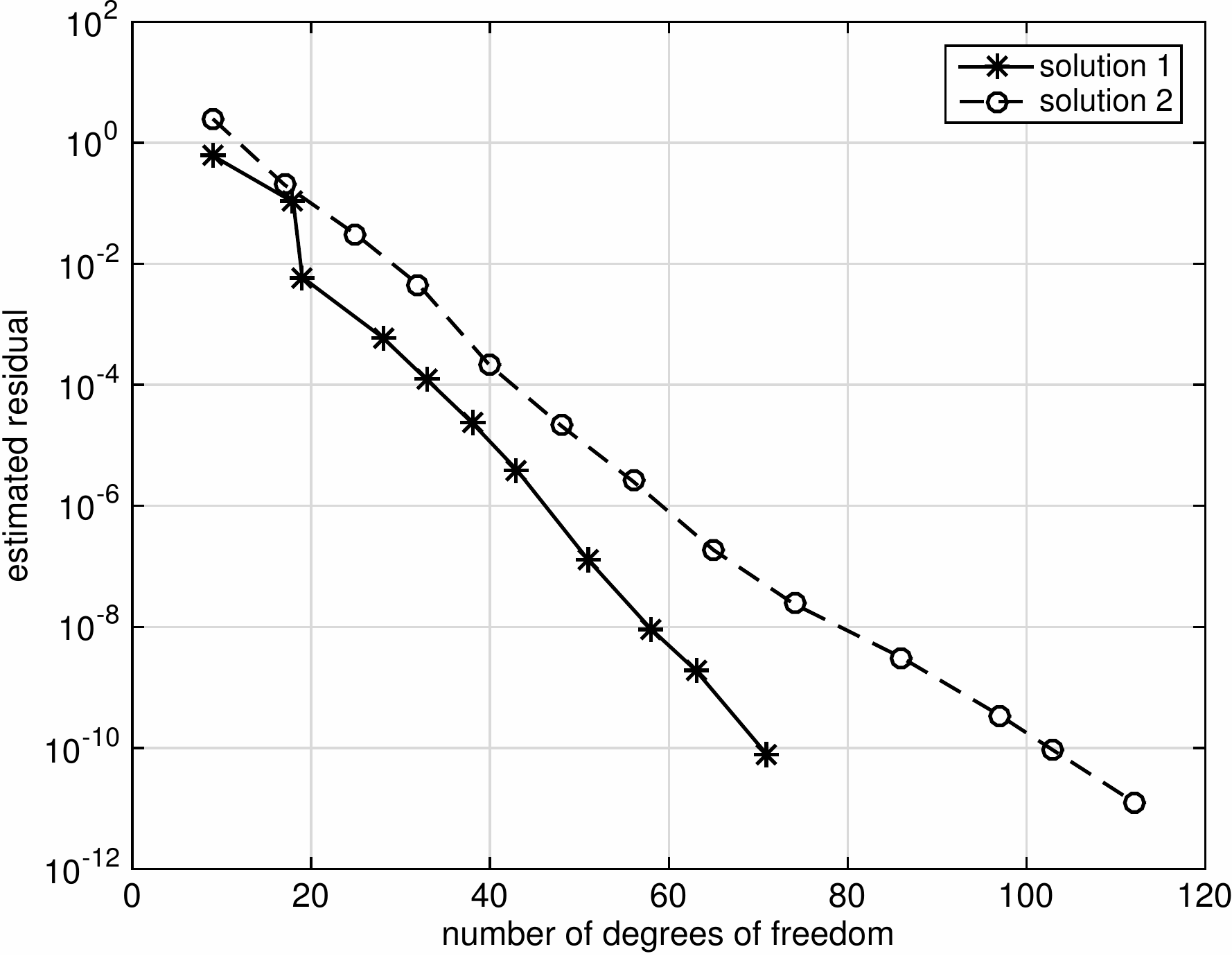}
\end{center}
\caption{Example~\ref{ex:slin0}: Two positive solutions of~\eqref{eq:Bratu} (left), and the corresponding performance data for the residuals (right).}
\label{fig:slin0}
\end{figure}

%%%%%%%%%%%%%%%%%%%%%%%%%%%%%%%%%%%%%%%%%%%%%

\begin{example}\label{ex:slin1}
The focus of the next sequence of experiments is on the robustness of the \emph{a posteriori} residual bound~\eqref{eq:upperbound} with respect to the singular perturbation parameter~$\e$ as~$\e\to 0$. To this end, let us consider the \emph{Ginzburg-Landau}-type equation,
\[
-\e u''=u-u^3\ \text{on} \ (0,1),\qquad u(0)=u(1)=0.
\]
Starting from the initial guess depicted in Figure~\ref{fig:slin1_1} (left), we test the fully adaptive Newton-Galerkin Algorithm~\ref{al:full} for different choices of~$\e=\{10^{-i}\}_{i=2}^5$. As~$\e\to 0$ the resulting solutions feature ever stronger boundary layers, as well as an interior shock close to the center of the domain; see Figure~\ref{fig:slin1_1} (right). The performance data in~Figure~\ref{fig:slin1_2} (left) shows that the residuals decay, firstly, fairly robust in~$\e$, and, secondly, exponentially fast with respect to the (square root) of the number of degrees of freedom (cf., e.g., \cite{Schwab98}). The final mesh for~$\e=10^{-5}$ is displayed in Figure~\ref{fig:slin1_2} (right); whilst the singular effects have been properly resolved by suitable $hp$-refined elements, we notice that the initial finite element space remains nearly unrefined in areas where the solution is approximately constant~$\pm 1$. 
%tauNewton = 0.1;
%tauSmoothness = 0.6;
%gamma = 0.5;
%thetaRefine = 0.5;
%thetaInterplay = 0.5;
%residualtol = .13e-12;
%uinit = @(x)sin(2*pi*x);
\eremk
\end{example}

\begin{figure}
\begin{center}
	\includegraphics[width=0.45\linewidth]{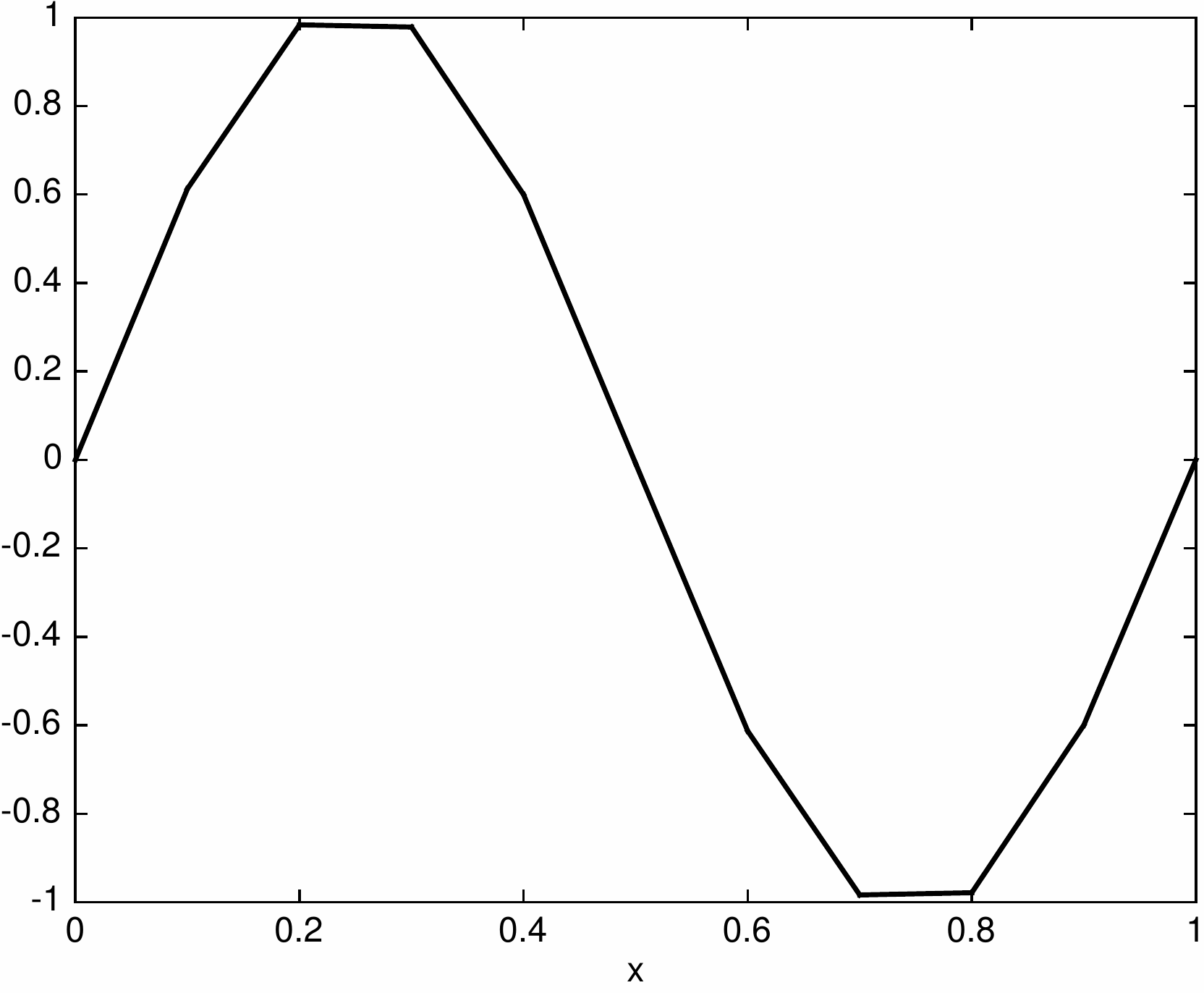}\hfill
	\includegraphics[width=0.45\linewidth]{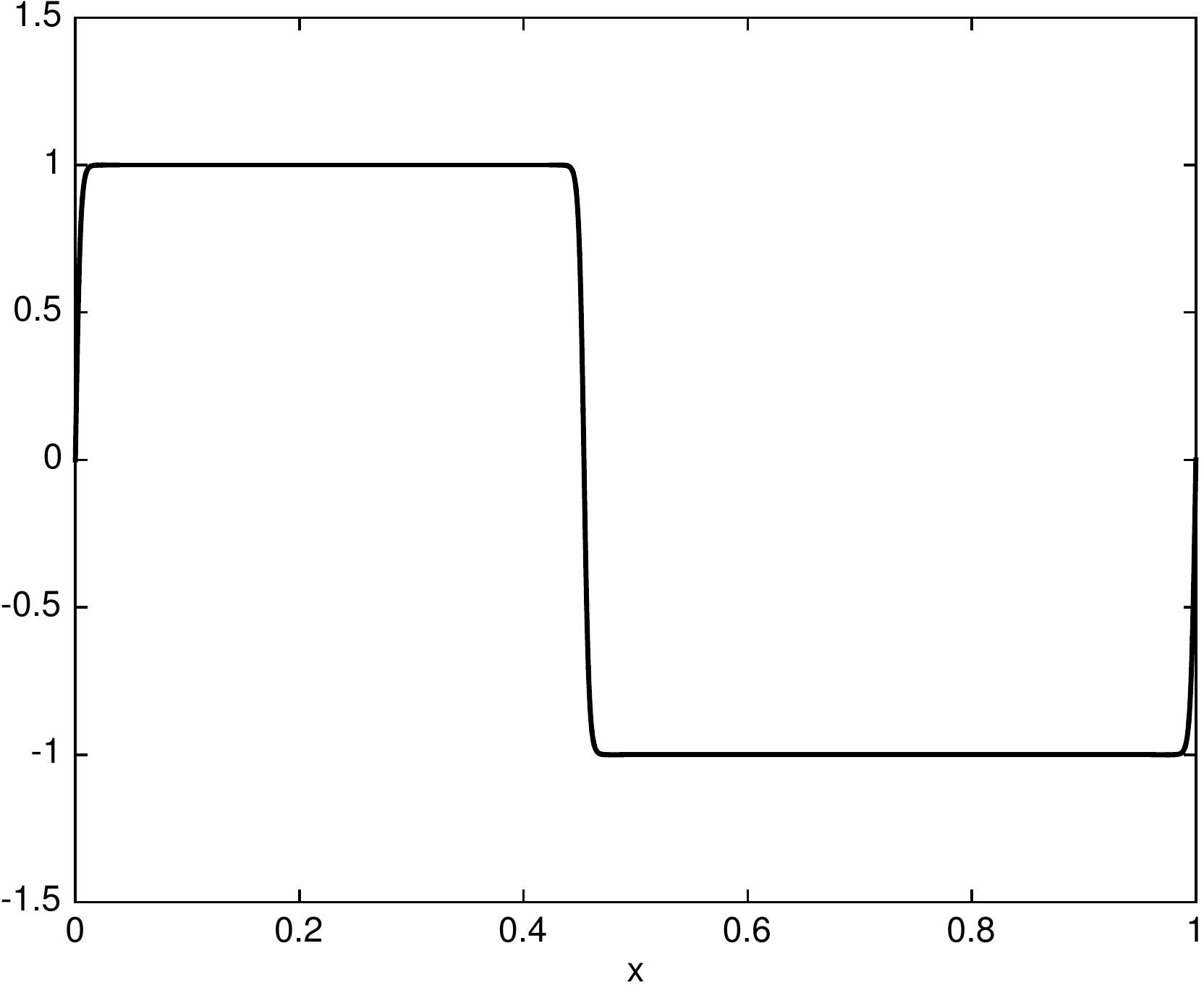}
\end{center}
\caption{Example~\ref{ex:slin1} for~$\e=10^{-5}$: Piecewise linear initial guess (left), and numerical solution with strong layers at the boundary and close to the center of the domain (right).}
\label{fig:slin1_1}
\end{figure}

\begin{figure}
\begin{center}
	\includegraphics[width=0.45\linewidth]{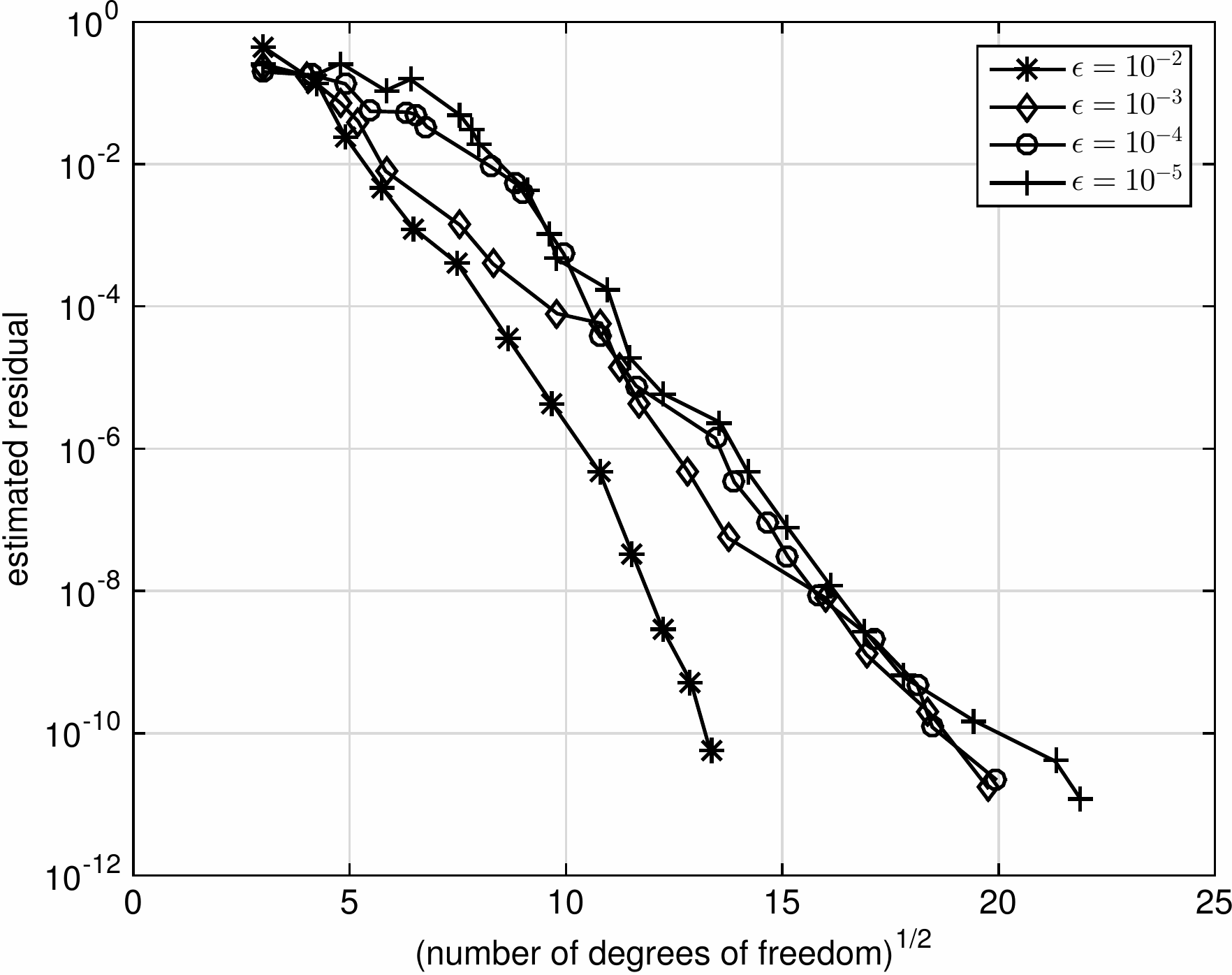}\hfill
	\includegraphics[width=0.45\linewidth,height=0.355\linewidth]{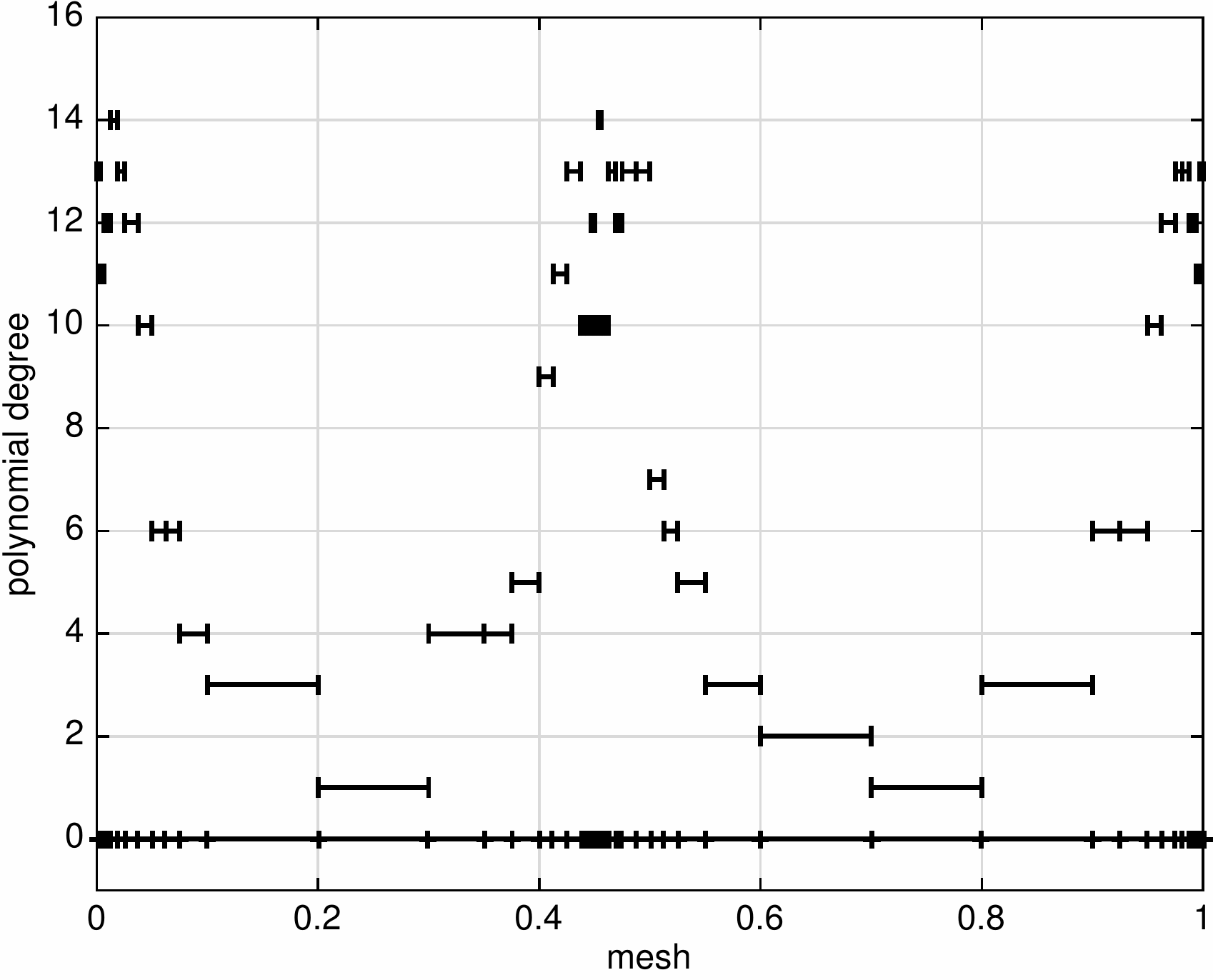}
\end{center}
\caption{Example~\ref{ex:slin1}: Estimated residuals for different choices of~$\e$ (left), and adaptively generated final $hp$-mesh for~$\e=10^{-5}$ (right).}
\label{fig:slin1_2}
\end{figure}

%%%%%%%%%%%%%%%%%%%%%%%%%%%%%%%%%%%%%%%%%%%%%

\begin{example}\label{ex:slin2}
Finally, we turn to the \emph{Fisher} equation,
\[
\begin{aligned}
-\e u'' =u-u^2 \ \text{on} \ (0,1),\qquad
u_{\e}(0)=\alpha,\quad \ u_{\e}(1)=\beta. 
\end{aligned}
\]
%
% 1: @(x) 1-(ul+x*(ur-ul));
% 2: @(x)(cos(3*pi*x).^2-(ul+x*(ur-ul)));
%
Solutions of this problem possess boundary layers, and, furthermore, for $ \alpha>-\nicefrac12 $ and  $\beta < 1 $, the solutions feature an increasing number of spikes (which are bounded between~0 and~1) as $\e \to 0$; see, e.g., \cite{Verhulst} for a more detailed discussion. The purpose of this example is to show that our fully adaptive Newton-Galerkin procedure is able to transport appropriate initial guesses sufficiently close along the corresponding trajectory of the Davydenko system~\eqref{eq:davy}, and, thereby, to obtain numerical solutions that retain the topological structure of the starting function. To proceed in this direction, we look at the two initial guesses displayed in Figures~\ref{fig:slin2_1}  and~\ref{fig:slin2_2} (left), as well as at the corresponding numerical solutions (right) resulting from Algorithm~\ref{al:full}. The initial mesh is based on 20 uniform elements in these experiments. We clearly see that the proposed method has generated two different solutions, which, essentially, show the same \emph{qualitative} properties as the respective initial guesses. We emphasize that, for~$\e\ll 1$, this will typically not work if a standard Newton approach is employed. Moreover, for the solution given in Figure~\ref{fig:slin2_2}, the $hp$-refined mesh is plotted in Figure~\ref{fig:slin2_3} (left), and shows that the boundary layers and interior spikes have been carefully resolved by the $hp$-adaptive scheme. In addition, from Figure~\ref{fig:slin2_3} (right), exponential convergence can be observed for both cases.\eremk

\begin{figure}
\begin{center}
	\includegraphics[width=0.45\linewidth]{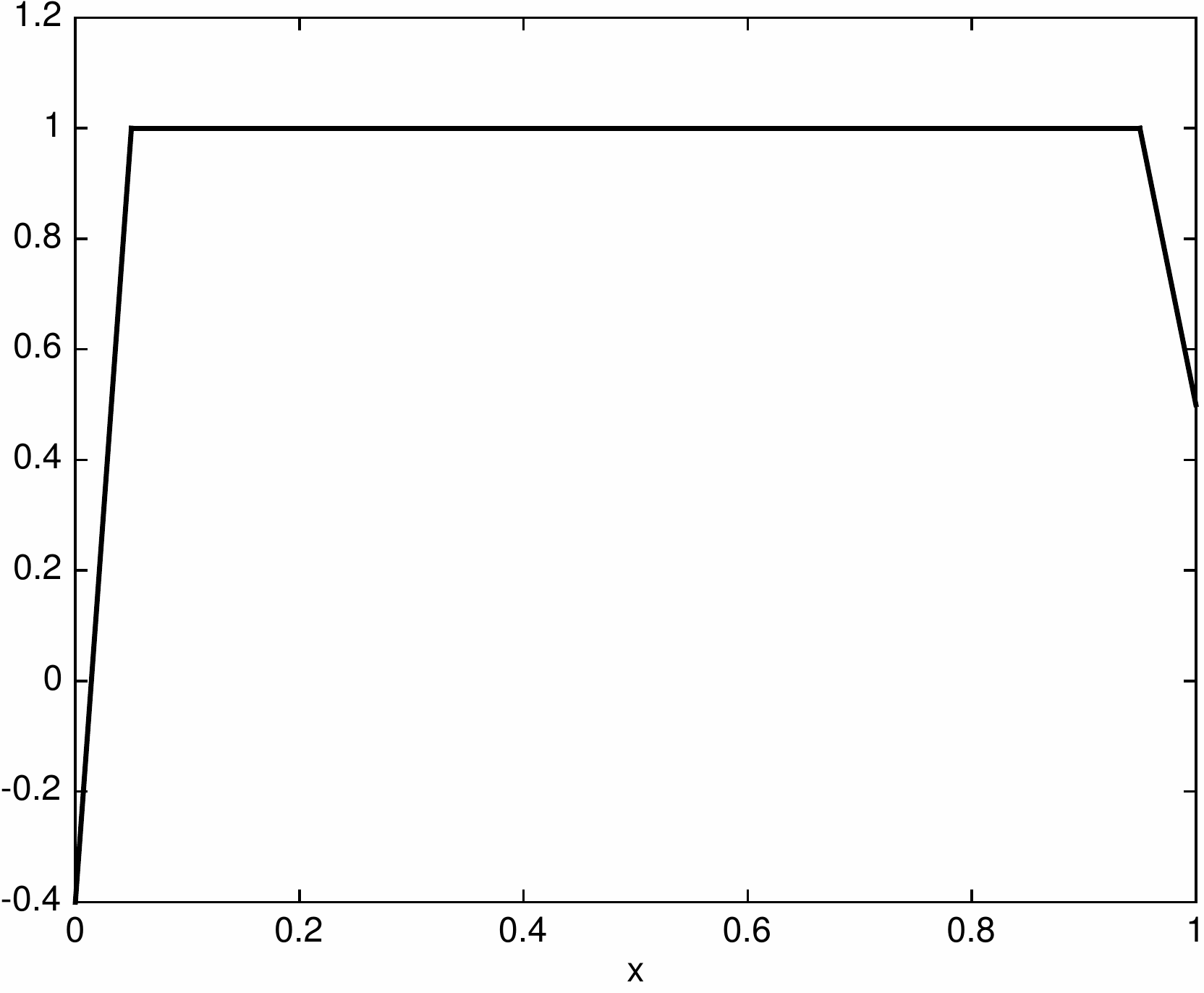}\hfill
	\includegraphics[width=0.45\linewidth]{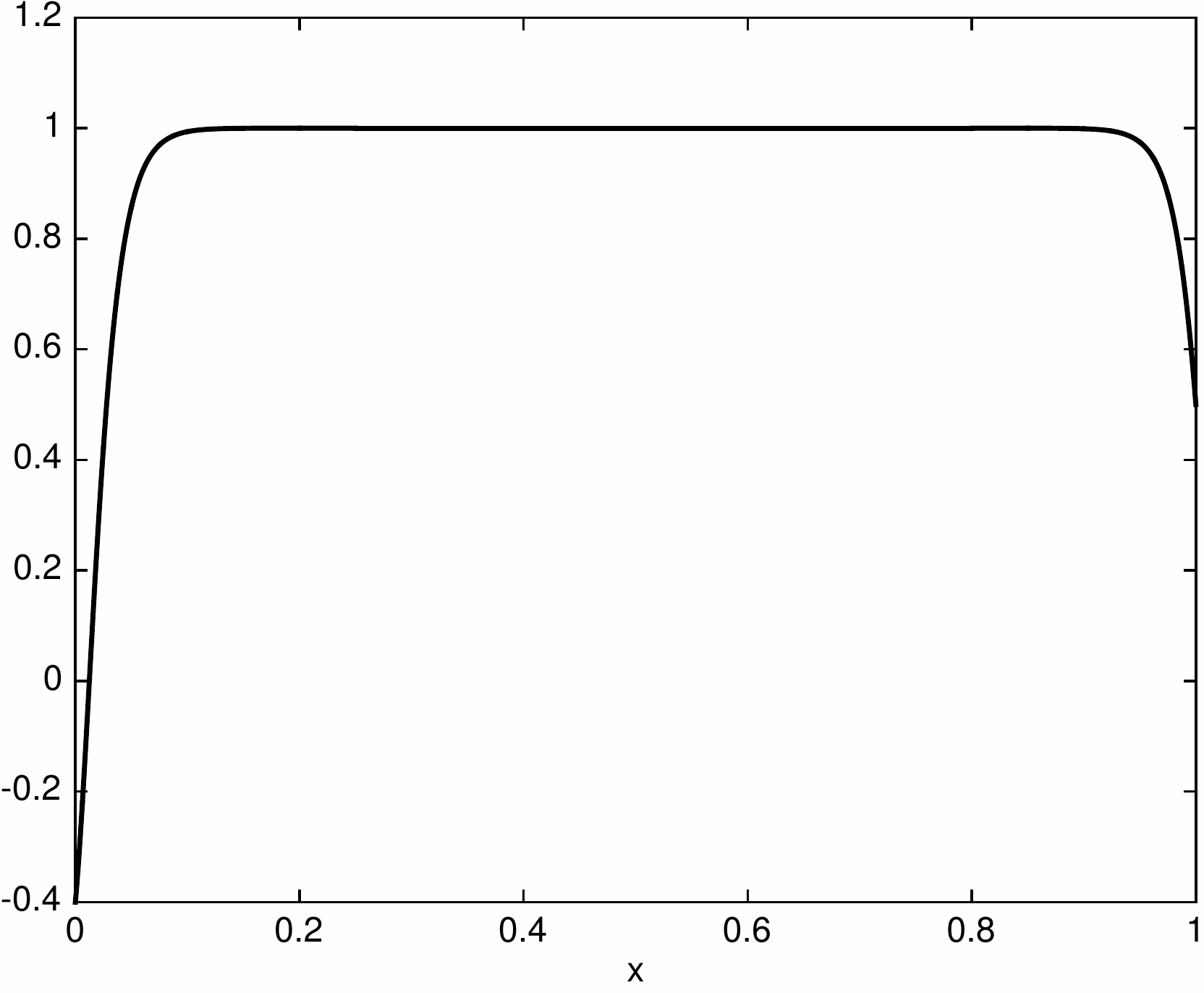}
\end{center}
\caption{Example~\ref{ex:slin2} for~$\e=0.00025$: Piecewise linear initial guess (left), and numerical solution (right).}
\label{fig:slin2_1}
\end{figure}

\begin{figure}
\begin{center}
	\includegraphics[width=0.45\linewidth]{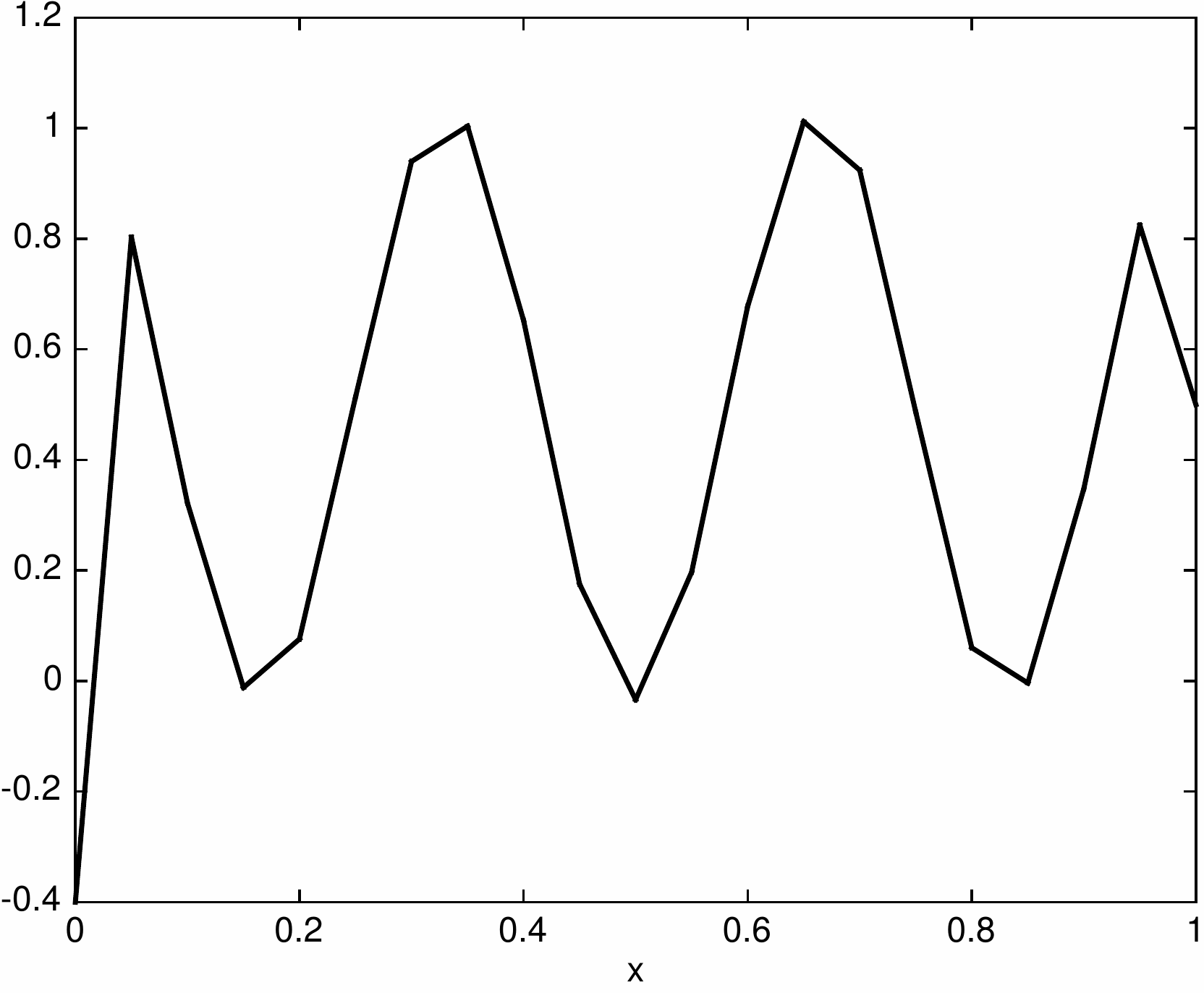}\hfill
	\includegraphics[width=0.45\linewidth]{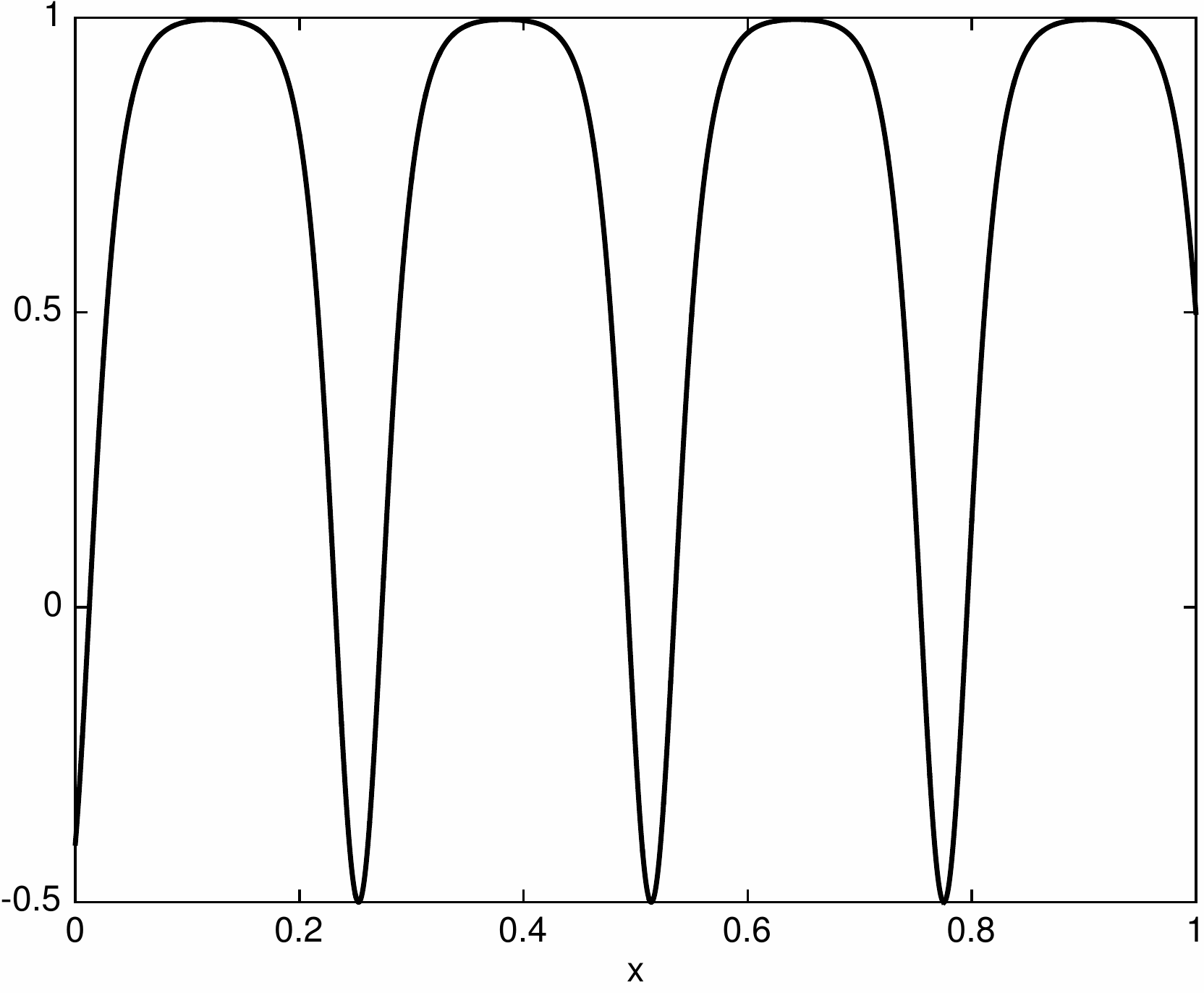}
\end{center}
\caption{Example~\ref{ex:slin2} for~$\e=0.00025$: Piecewise linear initial guess (left), and numerical solution (right).}
\label{fig:slin2_2}
\end{figure}

\begin{figure}
\begin{center}
	\includegraphics[width=0.45\linewidth]{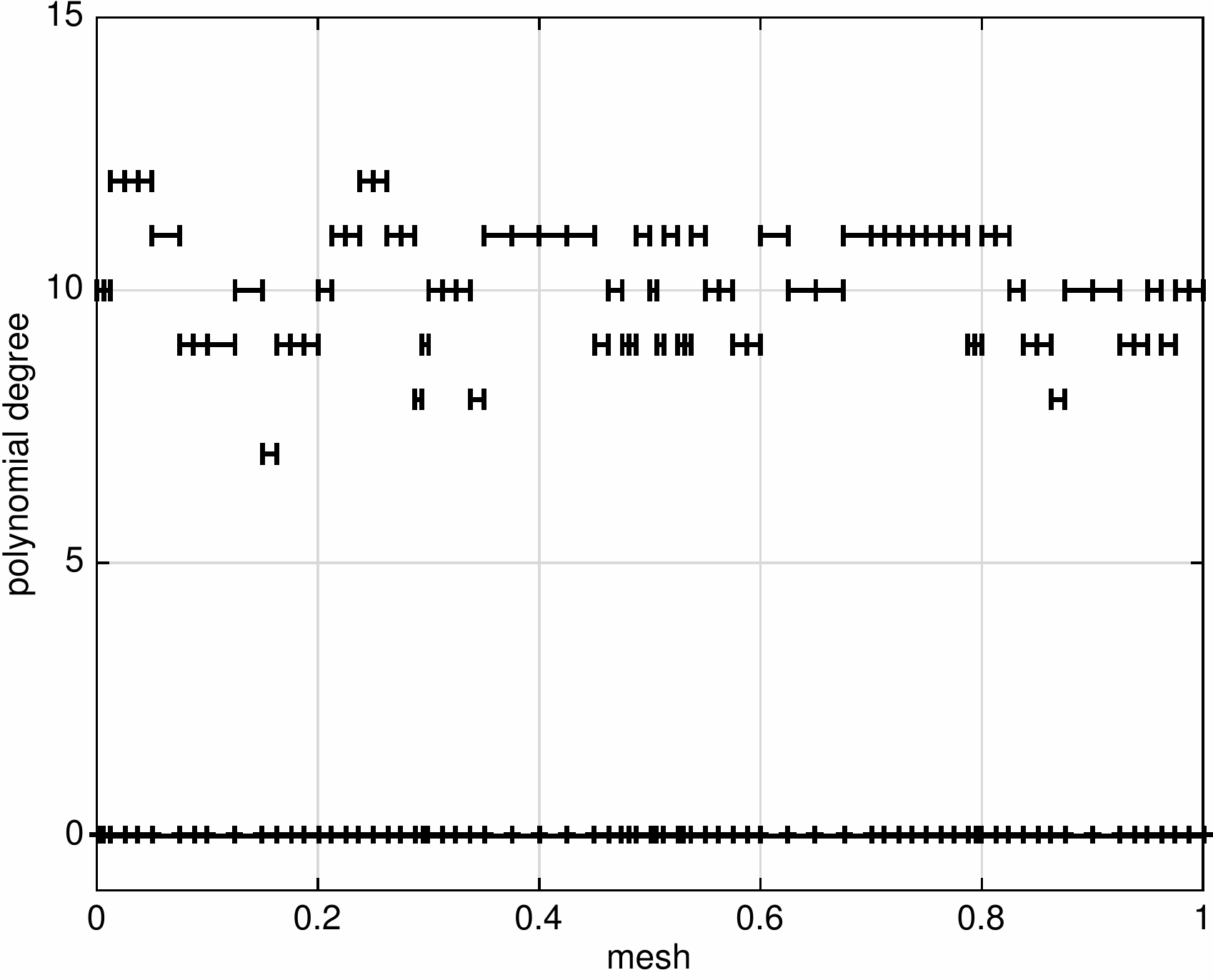}\hfill
	\includegraphics[width=0.45\linewidth,height=0.365\linewidth]{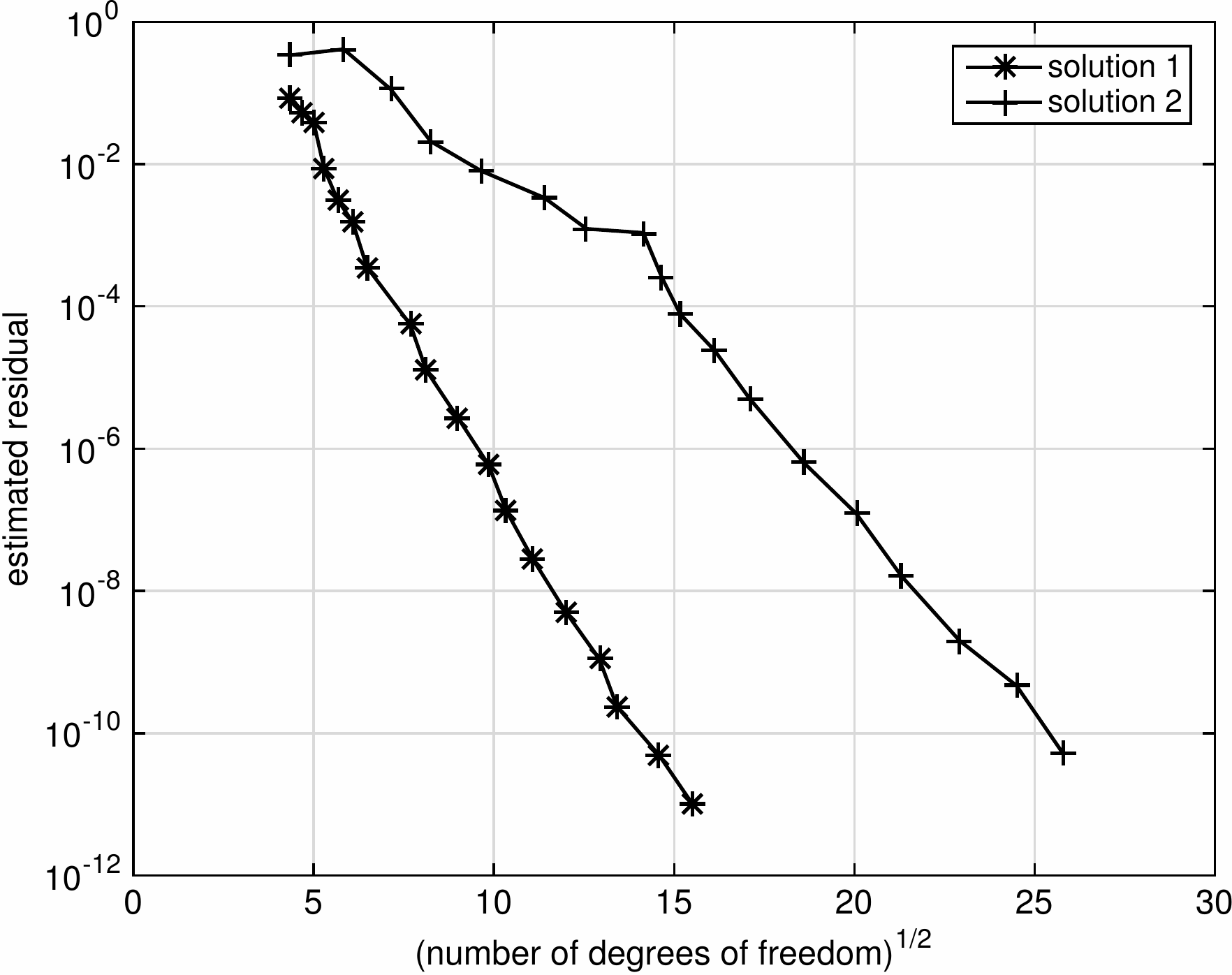}
\end{center}
\caption{Example~\ref{ex:slin2} for~$\e=0.00025$: Final $hp$-refined mesh (left), and exponential decay of residuals for the two solutions from Figure~\ref{fig:slin2_1} and~Figure~\ref{fig:slin2_2} (right).}
\label{fig:slin2_3}
\end{figure}
\end{example}

\section{Conclusions}\label{sc:concl}
The aim of this paper was to introduce a reliable and computationally feasible procedure for the numerical solution of general, semilinear elliptic boundary value problems with possible singular perturbations. The key idea is to combine an adaptive Newton method with an automatic mesh refinement finite element procedure. Here, the (local) Newton damping parameter is selected based on interpreting the scheme within the context of step size control for dynamical systems. Furthermore, the sequence of linear problems resulting from the Newton discretization is treated by means of a robust (with respect to the singular perturbations) $hp$-version {\em a posteriori} residual analysis, and by a corresponding adaptive $hp$-refinement scheme. The numerical experiments clearly illustrate the ability of our approach to reliably find solutions reasonably close to the initial guesses (in the sense that solutions with similar qualitative structure like the initial guesses can be found), and to robustly resolve the singular perturbations at an exponential rate.

\bibliographystyle{amsplain}
\bibliography{references}

\end{document}